\documentclass[a4paper,12pt]{amsart}
\usepackage{amssymb}
\usepackage{latexsym}
\usepackage{amsmath}
\usepackage[english]{babel}
\usepackage{amsfonts}
\usepackage{enumerate}
\usepackage[all]{xy}
\CompileMatrices

\font\sc=rsfs10
\newcommand{\cC}{\sc\mbox{C}\hspace{1.0pt}}

\newcommand{\Hom}{\operatorname{Hom}}
\newcommand{\Ext}{\operatorname{Ext}}
\newcommand{\Tor}{\operatorname{Tor}}

\newcommand{\ch}{\operatorname{ch}}

\usepackage{bbm} 
\newcommand\C{\mathbb C}
\newcommand\Z{\mathbb Z}

\newcommand\N{\mathbb N}
\newcommand\Q{\mathbb Q}

\usepackage{amsthm}
\newtheorem{thm}{Theorem}[section]
\newtheorem{exm}[thm]{Example}
\newtheorem{prop}[thm]{Proposition}
\newtheorem{cor}[thm]{Corollary}
\newtheorem{lemma}[thm]{Lemma}

\theoremstyle{definition}

\newtheorem{conj}[thm]{Conjecture}

\numberwithin{equation}{section}

\title{Category $\mathcal{O}$ for quantum groups}
\author{Henning Haahr Andersen and Volodymyr Mazorchuk}

\begin{document}

\begin{abstract}
In this paper we study the BGG-categories $\mathcal O_q$ associated to quantum groups. 
We prove that many properties of the ordinary  BGG-category $\mathcal O$ for a 
semisimple complex Lie algebra carry over to the quantum case. 

Of particular interest is the case when $q$ is a complex root of unity. Here we prove a tensor
decomposition for both simple modules, projective modules, and indecomposable tilting modules. 
Using the known Kazhdan-Lusztig conjectures for $\mathcal O$ and for finite dimensional 
$U_q$-modules we are able to determine all irreducible characters as well as the characters of 
all indecomposable tilting modules in $\mathcal O_q$. 

As a consequence of these results we are able to recover also a known result, namely 
that the generic quantum case behaves like the classical category $\mathcal O$.
\end{abstract}

\maketitle

\section{Introduction}\label{s1}

Let $\mathfrak g$ be a semi-simple Lie algebra over $\Q$ associated to a semi-simple complex Lie algebra
$\mathfrak g_{\mathbb{C}}$. 
The corresponding BGG-category $\mathcal O$, defined in \cite{BGG}, has been studied intensively over the 
last decades, see the recent monograph \cite{Hu} for details. 

In this paper we study similar categories for quantum groups. We let $v$  denote an indeterminate 
and set $U_v$ equal to the quantum group (or rather the quantized enveloping algebra) for 
$\mathfrak g$ over $\Q(v)$. The subcategory $\mathcal O_v$ of the module category for $U_v$ is then 
defined in complete analogy with $\mathcal O_{int}$, the subcategory of $\mathcal O$ consisting of modules 
with integral weights. Using Lusztig's quantum divided power version of $U_v$ (see below) it is possible to consider 
also specializations $\mathcal O_q$ of $\mathcal O_v$ for
any $q \in \C\setminus\{0\}$. Our principal interest in this paper is the study of $\mathcal O_q$ in
the case where $q$ is a root of unity.

In a little more detail set $A = \Z[v, v^{-1}] \subset \Q(v)$ and let $U_A$ be the Lusztig $A$-form of $U_v$, cf. \cite{Lu90b}.
For any non-zero $q \in \C$ we set $U_q = U_A \otimes_A \C$, where $\C$ is made into an $A$-algebra by 
the specialization $v \mapsto q$. Then $\mathcal O_q$ is the BGG-category for $U_q$. 

Denote by $\mathcal F_q$ the subcategory of $\mathcal O_q$ consisting of all finite dimensional $U_q$-modules
of type $\bf 1$. When $q$ is not a root of unity, this category is semisimple and it has exactly the
same ``combinatorics'' as the category of finite dimensional modules for 
$\mathfrak g_{\mathbb{C}}$, i.e. the characters of the simple modules are given by Weyl's character formula, see e.g. 
\cite{APW91,Ja}. 

Suppose now $q$ is a root of unity. Then $\mathcal F_q$ has a much more complicated structure.
In the present paper we assume that $q$ is of odd order $l$ and if $\mathfrak g$ contains 
a summand of type $G_2$ we assume in addition that $l$ is prime to $3$. Then the ``combinatorics'' of $\mathcal F_q$ has been worked out: Lusztig stated
the conjecture \cite{Lu89} that irreducible characters in $\mathcal F_q$ should be given by the values at $1$ of the Kazhdan-Lusztig
polynomials associated to the affine Weyl group for $\mathfrak g$. More specifically, the Kazhdan-Lusztig polynomials involved are the
parabolic ones corresponding to affine Weyl group for the Langlands dual of
$\mathfrak g$ relative to its finite Weyl subgroup, cf also \cite{So1}. Kazhdan and Lusztig proved that the category of finite dimensional modules (of type {\bf 1}) 
for $U_q$ is equivalent to a category of modules for the corresponding affine Kac-Moody algebra \cite{KL94}. This requires a weak restriction on $l$ and 
in the non-simply laced case it has to be supplemented 
by Lusztig's later work \cite{Lu94}. Then Kashiwara and Tanisaki proved the corresponding conjecture for affine Kac-Moody algebras, 
\cite{KT95} and \cite{KT96}. This established the above mentioned conjecture on the irreducible characters in $\mathcal F_q$.  
Soergel subsequently determined the characters of indecomposable tilting modules in $\mathcal F_q$, see  \cite{So2} and \cite{So3}.

We prove that several classes of fundamental modules such as simple modules, indecomposable projective modules, indecomposable injective modules, and indecomposable tilting modules
in $\mathcal O_q$ have a tensor product decomposition with a part which ``comes from'' $\mathcal F_q$
and a part which is a $q$-Frobenius twist of a corresponding module in $\mathcal O_{int}$, see Sections~\ref{s3} and~\ref{s4} for the 
precise statements. This allows us to deduce our main results: we determine all irreducible characters as well as the characters of all 
indecomposable tilting modules in $\mathcal O_q$, see Corollaries 5.3-4. 

In the process of establishing these results we prove that many of the properties of $\mathcal O$, e.g.
finite length of all modules, the existence of enough projectives and injectives, existence of tilting modules, and Ringel self-duality all carry over to $\mathcal O_q$.

One of the main features of $\mathcal O_q$ is that it contains a copy of $\mathcal O_{int}$, 
namely we may identify $\mathcal O_{int}$ with the direct sum of all ``special blocks'' 
in $\mathcal O_q$, see Theorem~\ref{thm9} below. Once we have established this and the 
above mentioned properties of $\mathcal O_q$ we return to the generic category $\mathcal O_v$. 
Using specialization of $v$ at $1$ on the one hand and at large order roots of unity $q$ 
on the other hand we are able to identify the combinatorics (i.e. the composition factor 
multiplicities of simple modules in Verma modules, and the multiplicities of Verma modules in 
Verma flags of indecomposable tilting modules) in $\mathcal O_v$ with that of $\mathcal O_{int}$. 
As was pointed out to us by D.~Kazhdan, in the simply laced case there is a stronger result,
proved by M. Finkelberg in his thesis \cite{Fi} from 1993, which establishes 
a category equivalence between $\mathcal O_q$ and $\mathcal O_{int}$ in the case when 
$q$ is a non-zero non-root of unity.  G.~Lusztig made us aware of the paper \cite{EK} where a 
generic equivalence is proved more generally for symmetrizable Kac-Moody algebras, see 
\cite[Theorem 4.2]{EK}. Our proof via the root of unity case is completely different.

The paper is organized as follows. In Section 2 we recall some basic facts about quantum groups at 
roots of unity. Then in Sections 3-4 we establish the results about $\mathcal O_q$ mentioned above. 
In particular, the tensor decompositions of simple modules, indecomposable projective or injective 
modules, and tilting modules are found in Theorem~\ref{prop3}, Theorem~\ref{Thm13}, Theorem~\ref{Thm15}, 
and Corollary~\ref{cor4.4.2}, respectively. Then we deduce in Section 5 the combinatorics of 
$\mathcal O_q$ before we conclude the paper in Section 6 by proving that in 
the generic case the combinatorics of $\mathcal O_v$ is the same as that of $\mathcal O_{int}$.
We complete the paper with a section which provides a parallel with the category $\mathcal{O}$ for Lie
superalgebras.
\vspace{5mm}

\noindent
{\bf Acknowledgements.} A major part of the research presented in the paper was done 
during the visit of the second author to the Center for Quantum Geometry of 
Moduli Spaces, Aarhus University. The financial support and 
hospitality of the Center are gratefully acknowledged.
The second author was partially supported by the
Royal Swedish Academy of Sciences and the Swedish Research Council.

We thank David Kazhdan and George Lusztig for calling our attention to related 
results already available in the literature. We also thank Jim Humphreys and the referee for very helpful
comments. We thank Eric Vasserot who made us aware of a mistake in the original formulation of Corollary~\ref{cor23}.

\section{Preliminaries on quantum groups}\label{s2}

\subsection{Quantum groups at roots of $1$}\label{s2.1}

For an indeterminate $v$ denote by $U_v$ the quantum group over $\Q(v)$ corresponding 
to a complex simple Lie algebra $\mathfrak g$. This is the $\Q(v)$-algebra with 
generators $E_i, F_i, K_i^{\pm 1}$, $i = 1, 2, \dots , n=\mathrm{rank}(\mathfrak g)$ 
and relations as given in \cite[Chapter~5]{Ja}.

Set $A = \Z[v, v^{-1}]$. Then $A$ contains the quantum numbers 
$[r]_d = \frac{v^{dr} - v^{-rd}}{v^d - v^{-d}}$ for any $r, d \in \Z$, $d \neq 0$ as 
well as the corresponding $q$ binomials $\left[\genfrac{}{}{0pt}{}{m}{t}\right]_d$, 
$m \in \Z, t \in \N$. When $r \geq 0$ we set $[r]_d! = [r]_d[r-1]_d \cdots [1]_d$. 
In the following we will often need these elements for $d = 1$ in which case we will 
omit it from the notation. 

Let $C$ be the Cartan matrix associated with $\mathfrak g$. We denote by $D$ a diagonal 
matrix whose entries are relatively prime natural numbers $d_i$ with the property that 
$D C$ is symmetric. Then we set $E_i^{(r)} = E_i^r/[r]_{d_i}!$. With a similar expression 
for $F_i^{(r)}$ we define now the $A$-form $U_A$ of $U_v$ to be the $A$-subalgebra of $U_v$ 
generated by the elements $E_i^{(r)}, F_i^{(r)}, K_i^{\pm 1}$, $i= 1, \dots, n$, $ r \geq 0$. 
This is the Lusztig divided power quantum group.

In this paper we fix throughout a primitive root of unity $q \in \C$ of odd order $l$. 
We assume that $l$ is prime to $3$ if $\mathfrak g$ has type $G_2$. The corresponding 
quantum group is then the specialization $U_q = U_A \otimes _A \C$ where $\C$ is considered 
as an $A$-module via $v \mapsto q$, cf. \cite{Lu90a}, \cite{Lu90b}. We abuse notation and write 
$E_i^{(r)}$ also for the element $E_i^{(r)} \otimes 1 \in U_q$ and similarly for
$F_i^{(r)}$.

We have a triangular decomposition $U_q = U_q^-U_q^0U_q^+$ with $U_q^-$ and $U_q^+$ being 
the subalgebra generated by $F_i^{(r)}$ or $E_i^{(r)}$, $i = 1, \dots , n$,  $r \geq 0$, 
respectively. The ``Cartan part'' $U_q^0$ is the subalgebra generated by $K_i^{\pm 1}$ and 
$\left[\genfrac{}{}{0pt}{}{K_i}{t}\right]$, $i = 1, \dots , n, \; t \geq 0$, where
\begin{displaymath}
\left[\genfrac{}{}{0pt}{}{K_i}{t}\right] = \prod _{j=1}^t 
\frac{K_iv^{d_i(1-j)} - K_iv^{-d_i(1-j)}}{v^{d_ij} - v^{-d_ij}}. 
\end{displaymath}
We denote the ``Borel subalgebra'' $U_q^0U_q^+$ by $B_q$.

Recall that $U_v$ is a Hopf algebra with comultiplication $\Delta$, counit $\epsilon$ 
and antipode $S$, see \cite[4.11]{Ja}. It is easy to see that their restrictions give $U_A$ 
the structure of a Hopf algebra over $A$. Then $U_q$ also gets an induced Hopf algebra structure. 

\subsection{The small quantum group}\label{s2.2}

We also have the small quantum group $u_q \subset U_q$, defined as the subalgebra of $U_q$ 
generated by $E_i, F_i, K_i^{\pm 1}$, $i = 1, \dots , n$. It is also a Hopf subalgebra. Note that 
$U_q$ is generated by $u_q$ and $E_i^{(l)}, F_i^{(l)}$, $i = 1, \dots , n$, as follows from \cite[Proposition 3.2(a)]{Lu89}. 

The small quantum group also has a triangular decomposition $u_q = u^-_qu_q^0u_q^+$ with the obvious 
definitions of the three parts. We write $b_q = u_q^0u_q^+$. Note that $u_q^-$ and $u_q^+$ are finite 
dimensional. In fact, the PBW basis for $U_q^-$ (resp. $U_q^+$) leads to a basis for $u_q^-$ 
(resp. $u_q^+$): we just have to take PBW-monomials where each ``root vector'' has degree at most 
$l$, \cite[Theorem 8.3]{Lu90b}. Also $u_q^0$ is finite dimensional. In fact, $K_i^{2l} = 1$ for all $i$, see \cite[5.7]{Lu90a}.

\subsection{The quantum Frobenius homomorphism}\label{s2.3}

Let $U_\C$ denote the enveloping algebra of $\mathfrak g$. It has generators $e_i, f_i$ and 
$h_i$, $i = 1, \dots , n$. Lusztig has then defined in \cite[Section 8]{Lu90b}, see also \cite[Part~V]{Lu},
a quantum Frobenius 
homomorphism  $Fr_q : U_q \rightarrow U_\C$ by

\begin{displaymath}
\begin{array}{lcl}
E_i^{(r)} \mapsto \begin{cases} e_i^{(r/l)} &\text { if } l 
\text { divides } r;\\0 & \text { if not.}\end{cases}, &\quad\quad & 
K_i \,\,\,\,\mapsto\,\, 1, \\
F_i^{(r)} \mapsto \begin{cases} f_i^{(r/l)} &\text { if } l 
\text { divides } r;\\0 & \text { if not.}\end{cases}, &\quad\quad & 
\left[\genfrac{}{}{0pt}{}{K_i}{t}\right] \mapsto  \left(\genfrac{}{}{0pt}{}{h_i}{t}\right).
\end{array}
\end{displaymath}
Here $ \left(\genfrac{}{}{0pt}{}{h_i}{t}\right)= \prod_{s=1}^{t} \frac{(h_i - s +1)}{s}. $
 
\subsection{Representations of $U_q$}\label{s2.4}
 
Set $X = \Z^n$. Then for $\lambda \in \Z^n$ we define $\chi_\lambda : U_q^0 \rightarrow \C$ by 
$\chi_\lambda (K_i^{\pm 1}) = q^{\pm d_i\lambda_i}$ and 
$\chi_\lambda (\left[\genfrac{}{}{0pt}{}{K_i}{t}\right]) = 
\left[\genfrac{}{}{0pt}{}{\lambda_i}{t}\right]_{d_i}$. This is a well-defined character of 
$U_q^0$ (see e.g. \cite[Lemma~1.1]{APW91}) and it extends to $B_q$ by mapping $E_i^{(r)}$ to $0$ for all $r > 0$, $i = 1, \dots , n$.
         
If $M$ is a $U_q^0$-module, then the $\lambda$ weight space of $M$ is defined as follows:
\begin{equation}\label{eq1}
M_\lambda = \{m \in M \mid u m = \chi_\lambda(u) m \text { for all } u \in U_q^0 \}. 
\end{equation}
The module $M$ is called a {\em weight module of type $\mathbf{1}$} provided that $M$ decomposes
into a direct sum of weight spaces of the form \eqref{eq1}. In this paper we consider only weight modules
of type $\mathbf{1}$ and will simply call them {\em weight modules}.
 
If $N$ is a $U_\C$-module then we may consider $N$ also as a $U_q$-module via $Fr_q$. To distinguish 
it from $N$ we denote this $U_q$-module by $N^{[l]}$ and  call it the ($q$-Frobenius) twist of $N$. 
Note that $u_q$ acts trivially on $N^{[l]}$. Conversely, if $M$ is a weight $U_q$-module on which 
$u_q$ acts trivially, then there exists a $U_\C$-module $N$ such that $M = N^{[l]}$, \cite[8.16]{Lu90b}.  
In this case we also write $N = M^{[-l]}$. Note that $N = M$ as $\C$-spaces and the action of $e_i$ (resp. $f_i$) 
on a vector $v \in N$ is given by $e_i v = E_i^{(l)} v$ (resp. $f_i v = F_i^{(l)} v$).

Note that $Fr_q$ restricts to homomorphisms $U_q^0 \rightarrow U_\C^0$ and $B_q \rightarrow B_\C$,
where $U_\C^0$ is the enveloping algebra of the Cartan subalgebra $\mathfrak h$ in $\mathfrak g$ generated by 
the $h_i$'s and $B_\C$ is the enveloping algebra of the Borel subalgebra of $\mathfrak g$ generated 
by the $h_i$'s and $e_i$'s. We also denote these homomorphisms by $Fr_q$. Using them we can twist 
$U_\C^0$- as well as $B_\C$- modules. For instance, the $1$-dimensional $U_q^0$- (or $B_q$-) 
module $\C_{l \lambda}$ is the twist of the $1$-dimensional $U_\C^0$- (or $B_\C$-) module 
$\C_\lambda$ determined by $\lambda \in X$ (we identify $X$ with the set of integral weights 
in $\mathfrak{h}^*$ in the usual way).

\section{The category $\mathcal O_q$}\label{s3}
 
\subsection{Definition}\label{s3.1}
 
Similarly to \cite{BGG} we define the category $\mathcal O_q$ as the full subcategory of $U_q$-mod 
consisting of those $U_q$-modules  $M$ which satisfy the following conditions:
\begin{enumerate}[$($I$)$]
\item\label{cond1} $M$ is finitely generated as a $U_q$-module,
\item\label{cond2} $M$ is a weight module,
\item\label{cond3} $\dim U_q^+ m < \infty$ for all $m \in M$.
\end{enumerate}
  
\begin{remark}\label{rem1}
Let $\mathcal{O}_{\mathrm{int}}$ denote the integral subcategory (i.e. the direct sum of all integral blocks) of 
the usual BGG category $\mathcal{O}$ for $\mathfrak g$ (see  \cite{BGG}). If $M\in \mathcal O_{\mathrm{int}}$ 
then $M^{[l]} \in \mathcal O_q$. 
\end{remark}
  
For $\lambda \in X$ the Verma $U_q$-module with highest weight $\lambda$ is given by the usual recipe:
\begin{displaymath}
\Delta_q(\lambda) = U_q \otimes _{B_q}\C_\lambda. 
\end{displaymath}
The standard arguments (see e.g. \cite[Chapter~7]{Di}) show that $\Delta_q(\lambda)$ has the 
following universal property:
\begin{displaymath}
\Hom_{\mathcal O_q} (\Delta_q(\lambda), M) = \{m \in M_\lambda \mid E_i^{(r)} m = 
0 \text { for all } r>0, i= 1, \dots, n \}.
\end{displaymath}

Moreover, it is easily seen that $\Delta_q(\lambda)$ has a unique maximal proper submodule. 
The corresponding simple quotient is denoted $L_q(\lambda)$. Then the set $\{L_q(\lambda):
\lambda \in X\}$ is a complete and irredundant set of representatives of isomorphism
classes of simple modules in $\mathcal O_q$.
  
\subsection{Infinitesimal modules}\label{s3.2}
  
Replacing $U_q$ by the small quantum group $u_q$ we get {\em baby Verma modules} defined by:
\begin{displaymath}
\bar \Delta_q(\lambda) = u_q \otimes_{b_q} \C_\lambda, \quad\quad \lambda \in X. 
\end{displaymath}
If we replace here $u_q$ by the subalgebra $u_qU^0_q$ of $U_q$ and $b_q$ by 
$U_q^0b_q = U_q^0u_q^+$, then we have similarly
\begin{displaymath}
\hat \Delta_q(\lambda) = u_qU_q^0 \otimes_{U_q^0b_q} \C_\lambda.
\end{displaymath}
The module $\hat \Delta_q(\lambda)$ restricted to $u_q$ coincides with 
$\bar \Delta_q(\lambda)$  and is a finite dimensional module. It has dimension $l^N$ where 
$N$ is the number of positive roots (because as a vector space we may identify it with $u^-_q$). 
The module $\hat \Delta_q(\lambda)$ has a universal property similar to the one enjoyed by 
$\Delta_q(\lambda)$ and it has a unique simple quotient which we denote $\hat L_q(\lambda)$.  
  
Set now $X_l = \{\lambda \in X \mid 0 \leq \lambda_i < l, i= 1, \dots , n\}$. 
Then each $\lambda \in X$ has an "$l$-adic expansion" $\lambda = \lambda^0 + l \lambda^1$ with 
$\lambda^0 \in X_l$, $\lambda^1 \in X$. In the following upper indices $0$ and $1$ on a weight 
will always refer to the components of the weight in this expansion.

We set $X^+ = \{\lambda \in X \mid \langle \lambda, \alpha^\vee \rangle \geq 0 \text { for all positive roots } \alpha \}$.
The elements of $X^+$ are called the {\em dominant weights}. An {\em  antidominant} weight is a $\lambda \in X$ for which $\lambda + \rho \in -X^+$.
  
We have the following remarkable fact about these infinitesimal simple modules, see \cite[Theorem 1.9]{AW}.
\begin{equation}\label{eq2}
\hat L_q(\lambda) \simeq L_q(\lambda^0) \otimes \C_{\lambda^1}^{[l]}. 
\end{equation}

The most ``special'' infinitesimal simple module is the one with highest weight $(l-1)\rho$. 
Here, as usual, $\rho$ is the half of the sum of all positive roots. We call this module the {\em quantum 
Steinberg module} and denote it by  $St_l$. Note that by \eqref{eq2} it is in fact a simple $U_q$-module,
moreover, we have
\begin{displaymath}
St_l = \hat L_q((l-1)\rho) = \hat \Delta_q((l-1)\rho) = L_q((l-1)\rho).
\end{displaymath}
  
\begin{remark}\label{rem2} 
Above we could also replace $u_q$ by $u_qB_q$. Then we get baby Verma modules for  $u_qB_q$ defined by 
$\tilde \Delta_q(\lambda) = u_qB_q \otimes_{B_q} \C_\lambda$ with simple quotient $\tilde L_q(\lambda)$. 
When restricted to $u_qU_q^0$ these modules coincide with $\hat \Delta_q(\lambda)$ and $\hat L_q(\lambda)$, 
respectively. Note, in particular, that the Steinberg module $St_l$ is also a simple $u_qB_q$-module 
as it extends, in fact, to $U_q$.

The composition factor multiplicities of $\tilde \Delta_q(\lambda)$ as well as the multiplicities with which 
$\tilde \Delta_q(\lambda)$ occurs in a baby Verma filtration of an indecomposable projective $u_qB_q$-module coincide with
the corresponding numbers for the Weyl module in $\mathcal F_q$ with highest weight $\lambda$
when $\lambda$ is sufficiently dominant. This follows from \eqref{eq2}, cf. \cite[Theorem 4.6]{APW92}. This fact allows us to
apply the combinatorics of $\mathcal F_q$ mentioned in the introduction to the category of $u_qB_q$-modules.
\end{remark}
   
\subsection{Tensor product formula for simple modules in $\mathcal O_q$}\label{s3.3}
  
Recall that $U_q$ is a Hopf algebra. In particular, its comultiplication allows us to make the tensor product 
(over $\C$) of two modules for $U_q$ into a $U_q$-module. 
We write $\otimes$ for $\otimes_{\C}$. Note that $\mathcal O_q$ is stable under tensoring with finite 
dimensional modules (of type $\bf 1$).
 
Let $M, N \in \mathcal O_q$. Then we consider $\Hom_\C(M, N)$ as a $U_q$-module in the usual way, see e.g. \cite[Section 2.9]{APW91}. 
The $u_q$-fixed points $\Hom_{u_q}(M, N)$ then form a $U_q$-submodule on which $u_q$ acts trivially, cf. \cite[Section 3.2]{APW92}. 
Hence by Section~\ref{s2.4} there exists a $U_\C$-module 
$P = \Hom_{u_q}(M, N)^{[-l]}$ with $P^{[l]} =  \Hom_{u_q}(M, N)$.

Let us also record the following adjointness valid whenever in addition to 
$M, N \in \mathcal O_q$ we have a module $Q \in \mathcal{O}$
\begin{equation}\label{eq3}
\Hom_{\mathcal O_q}(Q^{[l]} \otimes M, N) \simeq \Hom_{\mathcal O}(Q, \Hom_{u_q}(M, N)^{[-l]}).
\end{equation} 
  
We now have the following theorem first proved by G. Lusztig in \cite[Theorem 7.4]{Lu89}:

\begin{thm}\label{prop3}
Let $\lambda \in X$. Then $L_q(\lambda) \simeq L_\C(\lambda^1)^{[l]} \otimes L_q(\lambda^0)$.
\end{thm}
   
\begin{proof}
Let $L$ be any simple $u_q$-module (of type $\bf 1$). Recall from the previous subsection 
that $L$ is the restriction of a simple $U_q$-module (which we also denote by $L$). Then for any 
$M \in \mathcal O_q$ the natural map $\Hom_{u_q}(L, M) \otimes L \to M$ which takes $f \otimes m$ 
to $f(m)$ is a $U_q$-ho\-mo\-mor\-phism. It is in fact an injection which identifies  
$\Hom_{u_q}(L, M) \otimes L$ with the $L$-isotypic component of the $u_q$-socle of $M$. 
By the above this $U_q$-module is equal to 
$N^{[l]}$ for some $N \in \mathcal O_{\mathrm{int}}$.
   
Applying these observations to $M = L_q(\lambda)$ we get therefore 
$L_q(\lambda) \simeq \Hom_{u_q}(L, L_q(\lambda)) \otimes L$ for some such $L$, i.e. 
$L_q(\lambda) \simeq L_1^{[l]} \otimes L$ with $L_1 \in \mathcal O$. Clearly $L_1$ must be 
irreducible, i.e. $L_1 = L_\C(\mu)$ for some $\mu \in X$. By Section 3.2 we have $L \simeq L_q(\nu)$ 
for some $\nu \in X_l$. By weight considerations and the uniqueness of the $l$-adic expansion of 
$\lambda$ we get $\mu = \lambda^1$ and $\nu = \lambda^0$.
\end{proof}  

\subsection{Verma modules in $\mathcal O_q$}\label{s3.4}

We now want to study the composition factors of Verma modules. If $M \in \mathcal O_q$ and $\mu \in X$, we denote by 
$[M:L_q(\mu)]$ the multiplicity of $L_q(\mu)$ as a composition factor of $M$. We use similar notation for modules in $\mathcal O$ and for $u_qB_q$-modules.

\begin{lemma}\label{lemma3b}
Let $M$ be a $B_\C$-module. Then the map 
\begin{displaymath}
u \otimes m \mapsto Fr_q(u) \otimes m, \; u \in U_q, \; m \in M
\end{displaymath}
is an isomorphism of $U_q$-modules 
\begin{displaymath}
U_q \otimes_{u_qB_q} M^{[l]} \simeq (U_\C \otimes _{B_\C} M)^{[l]}.
\end{displaymath}
\end{lemma}

\begin{proof}
The map is clearly both well defined and a $U_q$-homomorphism. Note that $u_qB_q = u_q^-B_q$ and that the restriction of $Fr_q$ to $U_q^-$ is a surjection onto $U_\C^-$ with kernel 
generated by the augmentation ideal of $u_q^-$. It follows that the two modules in question are both isomorphic as 
$\C$-spaces to $U_\C^- \otimes M$ 
with the claimed map identifying the two.
\end{proof}

\begin{prop}\label{prop3c}
For $\lambda \in X$ the Verma module $\Delta_q(\lambda)$ has a filtration in $\mathcal O_q$ with quotients
of the form $\Delta_\C(\mu^1)^{[l]} \otimes \tilde L_q(\mu^0)$, $\mu \in X$. Each quotient $\Delta_\C(\mu^1)^{[l]} \otimes \tilde L_q(\mu^0)$
occurs $[\tilde \Delta_q(\lambda):\tilde L_q(\mu)]$ times.
\end{prop}

\begin{proof}
Consider a composition series of $\tilde \Delta_q(\lambda)$
\begin{displaymath}
0 = F^r \subset F^{r-1} \subset \cdots \subset F^0 = \tilde \Delta_q(\lambda)
\end{displaymath}
with quotients $F^{i-1}/F^i \simeq \tilde L_q(\mu_i)$. When we apply the exact functor $U_q \otimes _{u_qB_q} {}_-$ 
we obtain the filtration 
\begin{displaymath}
0 = U_q \otimes _{u_qB_q} F^r \subset U_q \otimes _{u_qB_q} F^{r-1} \subset \cdots \subset U_q \otimes _{u_qB_q} F^0 
\end{displaymath}
of $U_q \otimes _{u_qB_q} \tilde \Delta_q(\lambda) \simeq
\Delta_q(\lambda)$ with quotients $U_q \otimes _{u_qB_q} \tilde L_q(\mu_i)$. Now we recall 
from Section \ref{s3.2} that for any $\mu \in X$ we have $\tilde L_q(\mu) \simeq \C_{l\mu^1} \otimes L_q(\mu^0)$. By the tensor identity we have
\begin{displaymath}
U_q \otimes _{u_qB_q} \tilde L_q(\mu) \simeq (U_q \otimes _{u_qB_q} \C_{l\mu}) \otimes L_q(\mu^0). 
\end{displaymath}
Finally, Lemma~\ref{lemma3b} shows that $U_q \otimes _{u_qB_q} \C_{l\mu} 
\simeq \Delta_\C(\mu^1)^{[l]}$ and the proposition follows.
\end{proof}

Recall that modules in $\mathcal O$ have finite 
composition series (see \cite[Chapter~7]{Di}). 
Moreover, by Proposition \ref{prop3} 
the composition factors of 
$\Delta_\C(\mu^1)^{[l]} \otimes L_q(\mu^0)$ are $L_\C(\nu^1)^{[l]} \otimes L_q(\mu^0) \simeq L_q(l \nu^1 + \mu^0)$ (occurring $[\Delta_\C(\mu^1) : L_\C(\nu^1)]$ times). We thus have

\begin{cor}\label{cor3d}
For every  $\lambda \in X$ the Verma module 
$\Delta_q(\lambda)$ has finite length. 
Moreover, for $\mu \in X$ we have
\begin{displaymath}
[\Delta_q(\lambda) : L_q(\mu)] = \sum_{\nu \geq \mu^1} [\tilde \Delta_q(\lambda) :\tilde L_q(l\nu + \mu^0)][\Delta_\C(\nu):L_\C(\mu^1)].
\end{displaymath}
\end{cor}

\begin{cor}\label{prop6}
All modules in $\mathcal O_q$ have finite length.
\end{cor}

\begin{proof} 
By Condition \eqref{cond1} of $\mathcal O_q$ it is enough to establish this for cyclic modules $M$, i.e. we assume $M = U_q m$ for some
$m \in M$. By Conditions \eqref{cond2} and \eqref{cond3}, $m$ is contained in a finite dimensional $B_q$-submodule $E \subset M$. This means that $M$ is a quotient of $U_q
\otimes_{B_q} E$ which has a finite Verma filtration (take a $B_q$-filtration of $E$ with $1$-dimensional quotients and apply the exact functor $U_q \otimes _{B_q} -$). It is therefore
enough to check that Verma modules in $\mathcal O_q$ have finite length. We did this in 
Corollary \ref{cor3d}.
\end{proof}

For later use we record the following consequence of Corollary \ref{cor3d}

\begin{cor}\label{cor3e}
Let $\lambda, \mu \in X$. Then for $l \gg 0$ we have 
\begin{displaymath}
[\Delta_q(\lambda) : L_q(\mu)] = [\tilde \Delta_q(\lambda): \tilde L_q(\mu)].
\end{displaymath}
\end{cor}

\begin{proof}
Choose $l$ so big that $\lambda - \mu \not \geq l\nu$ for any $\nu > 0$. Then the sum on the right hand side of the formula in Corollary \ref{cor3d} contains only 
one term, namely the term with $\nu = \mu^1$.
\end{proof}

\subsection{Special modules in $\mathcal O_q$}\label{s3.5}
     
\begin{prop}\label{prop4}
Let $\lambda \in X$. Then 
$\Delta_q(l \lambda + (l-1)\rho) \simeq  \Delta_\C(\lambda)^{[l]} \otimes St_l$.
\end{prop}
   
\begin{proof}
We have $\tilde \Delta_q((l-1)\rho) \simeq St_l$, see \cite[Lemma~2.6]{APW92}. Just as in the proof of Proposition \ref{prop3c} we then get
\begin{multline*}
\Delta_q(l \lambda + (l-1)\rho) \simeq U_q \otimes _{u_qB_q} \tilde \Delta_q(l\lambda + (l-1)\rho)
\simeq\\ (U_q \otimes _{u_qB_q} \otimes \C_{l \lambda}) \otimes St_l \simeq \Delta_\C(\lambda)^{[l]} \otimes St_l.
\end{multline*}
\end{proof}

\begin{cor}\label{cor5}
If $\lambda$ is antidominant then $\Delta_q(l \lambda + (l-1)\rho)$ is simple. 
In particular, $\Delta_q(-\rho)$ is simple.
\end{cor}

\begin{proof}
It is well known (see e.g. \cite[Chapter~7]{Di} or \cite[Theorem~4.4]{Hu}) that $\Delta_\C(\lambda)$ is simple 
in $\mathcal O$  when $\lambda$ is antidominant.
\end{proof}

\subsection{The special block in $\mathcal O_q$}\label{s3.6}

The considerations at the beginning of Section~\ref{s3.3} allow us to define a functor $\mathrm{F} : \mathcal O_q \to \mathcal{O}_{\mathrm{int}}$ by 
\begin{displaymath}
\mathrm{F}\, N = \Hom_{u_q}(St_l, N)^{[-l]} .
\end{displaymath}
Since $St_l$ is projective as a $u_q$-module $\mathrm{F}$ is exact. 

Note that the map $f \otimes s \mapsto f(s)$ is a homomorphism and in fact an inclusion 
$(\mathrm{F}\,  N)^{[l]} \otimes St_l \to N$.  The considerations in Section~\ref{s3.3} prove
the following:

\begin{prop}\label{prop7}
Let $\lambda \in X$. Then
\begin{displaymath}
\mathrm{F}\, L_q(\lambda) \simeq 
\begin{cases} L_\C(\lambda^1), &\text { if } \lambda = l\lambda^1 + (l-1)\rho; \\ 0,
&\text { if } l \text { does not divide }\lambda + \rho.
\end{cases}
\end{displaymath}
\end{prop}

This is a key ingredient in the following:

\begin{prop}\label{prop8}
Let $\lambda, \mu \in X$. Suppose $\lambda^0 = (l-1)\rho \neq \mu^0$. 
Then $\Ext^i_{\mathcal O_q}(L_q(\lambda), L_q(\mu)) = 0$ for all $i$.
\end{prop}

\begin{proof}
As $L_q(\lambda) = L_\C(\lambda^1)^{[l]} \otimes St_l$ and $St_l$ is projective 
as a $u_q$-module, for any $M\in \mathcal O_q$ we get via \eqref{eq3}
\begin{displaymath}
\Ext^i_{\mathcal O_q}(L_q(\lambda), M) \simeq \Ext^i_{\mathcal O}(L_\C(\lambda^1), \mathrm{F}\, M). 
\end{displaymath}
When $M = L_q(\mu)$ we have $\mathrm{F}\, M = 0$ by Proposition~\ref{prop7} and the desired vanishing follows.
\end{proof}

Proposition~\ref{prop8} allows us to define $\mathcal O_q^{\mathrm{spe}c}$ to be the block in $\mathcal O_q$ 
consisting of those $M \in \mathcal O_q$ whose composition factors all belong to $l X +(l-1)\rho$.  
We call this the {\em special block} in $\mathcal O_q$ and its objects {\em special modules} in $\mathcal O_q$.

Define the functor $\mathrm{G}: \mathcal O \to \mathcal O_q^{\mathrm{spec}}$ by
\begin{displaymath}
\mathrm{G}\,N = N^{[l]} \otimes St_l.
\end{displaymath}
Note that for $N \in \mathcal O$ we have indeed that $\mathrm{G}\,N$ is a special module in $\mathcal O_q$. 

Clearly, $\mathrm{G}$ is exact and is in fact adjoint (left and right) to $\mathrm{F}$.
It is also immediate that $\mathrm{F} \circ \mathrm{G}$ is isomorphic to the identity functor on $\mathcal O$.  
Moreover, by Theorem~\ref{prop3} we have that $\mathrm{G} \circ \mathrm{F}$ is naturally isomorphic to
the identity on simple modules and hence on $\mathcal O_q^{\mathrm{spec}}$. We have thus proved the following: 

\begin{thm}\label{thm9}
There is  an equivalence of categories $\mathcal{O}_{\mathrm{int}} \cong \mathcal O_q^{\mathrm{spec}}$ given by 
the mutually inverse functors $\mathrm{F}$ and $\mathrm{G}$.
\end{thm}

\subsection{Projective modules in $\mathcal O_q$}\label{s.3.7}

Recall that in $\mathcal O_{int}$ the Verma module $\Delta_\C(\lambda)$ is projective whenever $\lambda + \rho$ is dominant,
cf. \cite[Proposition 3.8]{Hu}.
Hence Theorem~\ref{thm9} gives

\begin{cor}\label{cor10}
If $\lambda +\rho $ is dominant, then $\Delta_q(l \lambda + (l-1)\rho)$ is projective in $\mathcal O_q$.
In particular, $\Delta_q(-\rho)$ is projective.
\end{cor}

More generally, let $\mu \in X$ and denote by $P_\C(\mu) \in \mathcal O$ a projective cover of $L_\C(\mu)$. Then 
Theorem~\ref{thm9} gives the following:

\begin{prop}\label{prop11}
For each $\lambda \in X$ the module $P_\C(\lambda)^{[l]} \otimes St_l$ is a projective cover of $L_q(l\lambda + (l-1)\rho)$ in $\mathcal O_q$.
\end{prop}

Having these projectives allows us to deduce the following:

\begin{thm}\label{Thm12}
The category $\mathcal O_q$ has enough projectives.
\end{thm}

\begin{proof}
This is a standard argument, cf. \cite[3.8]{Hu}: By induction with respect to length we reduce the problem to proving that each simple module 
can be covered by a projective. Given $\lambda \in X$, we set $\nu = w_0\lambda^0 +(l-1)\rho$ where $w_0$ denotes the longest element in the 
Weyl group $W$ for $\mathfrak g$. Then $w_0 \nu = \lambda^0 -(l-1)\rho$ is the lowest weight of the finite dimensional simple module $L_q(\nu)$. 
Therefore $\Delta_q(l\lambda^1 +(l-1)\rho) \otimes L_q(\nu)$ surjects onto $\Delta_q(l\lambda^1 +(l-1)\rho +w_0\nu) = \Delta_q(\lambda)$. 
Now it is an easy consequence of Proposition~\ref{prop11} that $P_\C(\lambda^1)^{[l]} \otimes St_l$ surjects onto $\Delta_q((l\lambda^1 +(l-1)\rho)$.
So we see that the projective module $P_\C(\lambda^1)^{[l]} \otimes St_l \otimes L_q(\nu)$ surjects onto $L_q(\lambda)$. 
\end{proof}

Define $P_q(\lambda) \in \mathcal O_q$ as the projective cover of $L_q(\lambda)$. Then Corollary~\ref{cor10} says that $P_q(\lambda) = \Delta_q(\lambda)$ for all $\lambda$ such
that $\lambda +\rho \in X^+ \cap lX$. Moreover, Proposition~\ref{prop11} says that $P_q(l\lambda + (l-1)\rho) \simeq P_\C(\lambda)^{[l]} \otimes St_l$ for all $\lambda
\in X$. We shall now generalize this by showing that all indecomposable projectives in $\mathcal O_q$ have a tensor factorization.

Recall that the subcategory $\mathcal F_q$ consisting of all finite dimensional modules in $\mathcal O_q$ also has enough projectives, see \cite[Section 4]{APW92}. Let us 
denote by $Q_q(\mu) \in \mathcal F_q$ the projective cover of $L_q(\mu)$ for $\mu \in X^+$. 

\begin{thm}\label{Thm13}
For any $\lambda \in X$ we have $P_q(\lambda) \simeq P_\C(\lambda^1)^{[l]} \otimes Q_q(\lambda^0)$.
\end{thm}

\begin{proof}
By \cite[Theorem 4.6]{APW92} the restriction to $u_q$ of $Q_q(\lambda^0)$ is the projective cover of $L_q(\lambda^0)$, i.e. for $\mu \in X_l$ we have 
\begin{displaymath}
\Hom_{u_q}(Q(\lambda^0), L_q(\mu)) = \begin{cases} \C &\text { if } \mu = \lambda^0;\\ 0, &\text { otherwise.}\end{cases}
\end{displaymath}
Hence, using \ref{eq3} and Theorem~\ref{prop3}, for any $\mu \in X$ we get
\begin{multline*}
\Ext^i_{\mathcal O_q}(P(\lambda^1)^{[l]} \otimes Q(\lambda^0), L_q(\mu)) = \\ \Ext_{\mathcal O_\C}^i(P_\C(\lambda^1), L_q(\mu^1) \otimes \Hom_{u_q}(Q(\lambda^0),
L_q(\mu^0))^{[-l]})
= \\ \begin{cases} \C &\text { if } \mu = \lambda \text  { and } i = 0;\\ 0, &\text { otherwise.}\end{cases}
\end{multline*} 

\end{proof}

Let us also record the following important consequence of the constructions in the proof of Theorem~\ref{Thm12}

\begin{cor}\label{cor13a}
Projective modules in $\mathcal O_q$ all possess Verma filtrations.
\end{cor}

\begin{proof}
Let $\lambda \in X$. We shall prove that the corollary holds for $P_q(\lambda)$. When $\lambda \in lX +(l-1)\rho$ this follows from the fact that the
corresponding statement is true in $\mathcal O$ combined with Theorem~\ref{thm9}. But then the result follows in general because the construction in 
the proof of Theorem~\ref{Thm12} reveals that $P_q(\lambda)$ may be obtained as a summand of a projective in $\mathcal O_q^{\mathrm{spec}}$ tensored by a finite dimensional 
module.
\end{proof}

\subsection{Injective modules in $\mathcal O_q$}\label{s.3.8}

Let $M$ be an arbitrary $U_q$-module. Since the antipode $S$ on $U_q$ is an antihomomorphism, the dual space $M^* = \Hom_\C(M, \C)$ has the  natural structure of a $U_q$-module
given by $uf(m) = f(S(u)m)$, $u\in U_q$, $f \in M^*$, $m \in M$. Now $U_v$ has an automorphism $\omega$ which interchanges $E_i$ and $F_i$ and inverts $K_i$ , see \cite[4.6]{Ja}.
Clearly, $\omega$ gives rise to an automorphism of $U_q$. Twisting $M^*$ by $\omega$ we get the $U_q$-module $^\omega M^*$ and when $M \in \mathcal O_q$ we set
\begin{displaymath}
M^{\star} = \oplus_{\lambda \in X} (^\omega M^*)_\lambda.
\end{displaymath}  
  
Then $({}_-)^{\star}$ is an endofunctor on $\mathcal O_q$, called {\em duality}, with the property that for each $\lambda \in X$ we have $\dim (M^{\star})_\lambda = \dim M_\lambda$. Hence $L^{\star}_q(\lambda) \simeq
L_q(\lambda)$ (i.e. $\star$ is {\em simple preserving}). The existence of $\star$ gives immediately:

\begin{thm}\label{Thm14}
 $\mathcal O_q$ has enough injectives.
\end{thm}
 
We set $I_q(\lambda) = P^{\star}_q(\lambda)$. This is the injective envelope of $L_q(\lambda)$ in $\mathcal O_q$ and if we denote by $I_\C(\mu)$ 
the injective envelope of $L_\C(\mu)$ in $\mathcal O_{int}$ then Theorem~\ref{Thm13} implies:
  
\begin{thm}\label{Thm15}
For any $\lambda \in X$ we have $I_q(\lambda) \simeq I_\C(\lambda^1)^{[l]} \otimes Q_q(\lambda^0)$.
\end{thm}
  
\subsection{Projective-injective modules in $\mathcal O_q$}\label{s.3.9}

By a projective-injective module we understand a module which is both projective and injective. We have

\begin{thm}\label{Thm3.9.1}
Let $\lambda\in X$. Then the following assertions are equivalent:
\begin{enumerate}[$($a$)$]
\item\label{Thm3.9.1-1} $P_q(\lambda) \simeq I_q(\lambda)$.
\item\label{Thm3.9.1-2} $L_q(\lambda)$ occurs in the socle of a projective-injective module in $\mathcal{O}_q$.
\item\label{Thm3.9.1-3} $L_q(\lambda)$ occurs in the top of a projective-injective module in $\mathcal{O}_q$.
\item\label{Thm3.9.1-4} $L_q(\lambda)$ occurs in the socle of some $\Delta_q(\mu)$, $\mu\in X$.
\item\label{Thm3.9.1-5} $\lambda$ is antidominant.
\end{enumerate}
\end{thm}

\begin{proof}
The corresponding statement for $\mathcal{O}_{\mathrm{int}}$ (and also for its parabolic subcategories) is 
well-known, see e.g. Addendum and Proposition~4.3 in \cite{Ir}. Hence
Theorem~\ref{thm9} implies the claim for $\lambda \in lX + (l-1)\rho$. 

Note that $Q_q(\lambda^0)$ is self-dual. Hence by Theorem \ref{Thm13} and \ref{Thm15} we see that (a) holds if and only if $P_\C(\lambda^1) \simeq I_\C(\lambda^1)$. 

Now it is clear that \eqref{Thm3.9.1-1} implies \eqref{Thm3.9.1-2} and \eqref{Thm3.9.1-2} implies \eqref{Thm3.9.1-3}. Because of Corollary~\ref{cor13a} we have that 
\eqref{Thm3.9.1-4} is a consequence of \eqref{Thm3.9.1-3}. Suppose $L_q(\lambda)$ is a submodule of
$\Delta_q(\mu)$ for some $\mu \in X$. Then Proposition \ref{prop3c} and Proposition~\ref{prop3} show that $L_\C(\lambda^1)$ is a submodule of $\Delta_\C(\nu)$ for some 
$\nu \in X$. By the $\mathcal O$-result this implies that $\lambda^1$ is antidominant. But this is equivalent to \eqref{Thm3.9.1-5}. Finally, \eqref{Thm3.9.1-3} 
implies \eqref{Thm3.9.1-1} by the observations in the beginning of the proof.
\end{proof}

The properties described in Theorem~\ref{Thm3.9.1} appear frequently in various categories associated with Lie
(super)algebras, e.g. see \cite[Theorem~48]{MS} and \cite[Theorems~6.1,6.2]{BS}. For Lie superalgebras this will
be further clarified in Section~\ref{s9}.

\section{BGG reciprocity, Struktursatz and Ringel self-duality}\label{s4}

\subsection{BGG reciprocity in $\mathcal O_q$}\label{s4.1}

The dual Verma module $\Delta^{\star}_q(\lambda)$ is denoted $\nabla_q(\lambda)$. Then we have the following easy but very useful vanishing theorem, cf. \cite[6.12]{Hu}.

\begin{thm}\label{Thm16}
Let $\lambda, \mu \in X$ be arbitrary. Then 
\begin{displaymath}
\Ext^i_{\mathcal O_q} (\Delta_q(\lambda), \nabla_q(\mu)) \simeq \begin{cases}\C 
&\text { if } i = 0 \text { and } \lambda = \mu;\\ 0 &\text { otherwise.}
\end{cases}
\end{displaymath}
\end{thm}

\begin{proof}
Note that $\Ext^i_{\mathcal O_q} (\Delta_q(\lambda), \nabla_q(\mu)) \simeq \Ext^i_{\mathcal O_q} (\Delta_q(\mu), \nabla_q(\lambda))$ by duality. This allows us to assume that $\lambda
\not < \mu$.
Easy weight arguments show that the theorem holds for $i=0$. Now by Corollary~\ref{cor13a} all projectives in $\mathcal O_q$ have Verma filtrations.
Moreover, we have a short exact sequence $0 \to K \to P_q(\lambda) \to \Delta_q(\lambda) \to 0$ with $K$  having a Verma filtration where all subfactors $\Delta_q(\lambda')$ have
$\lambda' > \lambda$. The $i>0$ part of the theorem follows 
then from this sequence by a dimension shift argument.
\end{proof}

As a consequence we see that if $M \in \mathcal O_q$ has a Verma (resp. dual Verma) filtration 
then the number of occurrences $(M:\Delta_q(\lambda))$ (resp. $(M:\nabla_q(\lambda))$ of
$\Delta_q(\lambda))$ (resp. $\nabla_q(\lambda))$ in this filtration equals the dimension of 
$\Hom_{\mathcal O_q}(M, \nabla_q(\lambda))$ (resp. $\Hom_{\mathcal O_q}(\Delta_q(\lambda), M)$). 
This immediately leads to the following BGG-reciprocity laws:

\begin{cor}\label{Cor17}
Let $\lambda, \mu \in X$. Then
\begin{displaymath}
 (P_q(\lambda):\Delta_q(\mu)) = [\Delta_q(\mu):L_q(\lambda)] = (I_q(\lambda): \nabla_q(\mu)).
 \end{displaymath}
 \end{cor}

In other words, the above means that $\mathcal{O}_q$ is a highest weight category in the sense
of \cite{CPS} (with infinitely many isomorphism classes of simple modules).

\subsection{The category $\cC$}\label{s4.2}

Let $\cC$ denote the full subcategory of $\mathcal{O}_q$ with objects $P_q(\lambda)$, $\lambda\in X$.
For simplicity we will identify objects of $\cC$ with elements in $X$. Then Proposition~\ref{prop6} implies
that $\cC$ is a locally finite dimensional $\mathbb{C}$-linear category (we refer to \cite{MOS} for 
generalities on representations of $\C$-linear categories). Moreover, from Proposition~\ref{prop6} and
Theorems~\ref{Thm12} and \ref{Thm14} it follows that for any $\lambda\in X$ there exists 
only finitely many $\mu\in X$ such that $\cC(\lambda,\mu)\neq 0$ and that  for any $\lambda\in X$ there exists 
only finitely many $\mu\in X$ such that $\cC(\mu,\lambda)\neq 0$.

Let $\cC\text{-}\mathrm{mod}$ (resp. $\mathrm{mod}\text{-}\cC$) denote the category of finite dimensional 
left (resp. right) $\cC$-modules, that is covariant (resp. contravariant) functors 
$\mathrm{M}:\cC\to \mathbb{C}\text{-}\mathrm{mod}$ (the latter being the category of  
finite dimensional complex vector spaces) satisfying $\sum_{\lambda\in X}\dim \mathrm{M}(\lambda)<\infty$. 
Then abstract nonsense  (see e.g. \cite{Ga}) implies that $\mathcal{O}_q$ is equivalent 
to $\mathrm{mod}\text{-}\cC$ and the latter is equivalent to $\cC\text{-}\mathrm{mod}$ by duality.

\subsection{Dominance dimension and Soergel's Struktursatz}\label{s4.3}

\begin{prop}\label{prop4.3.1}
The category  $\mathcal{O}_q$ has dominance dimension at least two with respect to projective-injective 
modules, that is for any projective module $P\in \mathcal{O}_q$ there exists an exact sequence
\begin{equation}\label{eq11}
0\to P\to X_1\to X_2, 
\end{equation}
where both $X_1$ and $X_2$ are projective-injective.
\end{prop}

\begin{proof}
This claim is well-known for $\mathcal{O}_{\mathrm{int}}$, see e.g. \cite[3.1]{KSX}. Hence
Theorem~\ref{thm9} implies the claim for $P\in \mathcal{O}_q^{\mathrm{spec}}$. By Theorem~\ref{Thm13},
every indecomposable projective can be obtained by tensoring an indecomposable projective from 
$\mathcal{O}_q^{\mathrm{spec}}$ with a finite dimensional module and taking direct summand. As this tensoring is both left
and right adjoint to an exact functor, it preserves projective-injective modules. Hence such tensoring
maps a sequence of the form \eqref{eq11} to a sequence of the form \eqref{eq11} and the claim follows.
\end{proof}

Denote by $\cC^{PI}$ the full subcategory of $\cC$ whose objects are all antidominant 
$\lambda\in X$, that is those $\lambda\in X$ for which 
the projective module $P_q(\lambda)$ is also injective (see Theorem~\ref{Thm3.9.1}). 
For $\lambda\in X$ define 
\begin{displaymath}
\mathrm{M}_{\lambda}:= \mathrm{Hom}_{\mathcal{O}_q}({}_-,P_q(\lambda))\in 
\mathrm{mod}\text{-}\cC^{PI} .
\end{displaymath}
Let $\overline{\cC}$ denote the full subcategory of $\mathrm{mod}\text{-}\cC^{PI}$ with objects 
$\mathrm{M}_{\lambda}$, $\lambda\in X$.

Define a functor $\Phi:\cC\to \overline{\cC}$ as follows: on objects we set 
$\Phi(\lambda):=\mathrm{M}_{\lambda}$, $\lambda\in X$; if $\lambda,\mu\in X$ and
$\varphi\in \mathrm{Hom}_{\mathcal{O}_q}(P_q(\lambda),P_q(\mu))$, then set
\begin{displaymath}
\Phi(\varphi):=\varphi\circ{}_-:
\mathrm{Hom}_{\mathcal{O}_q}({}_-,P_q(\lambda))\to 
\mathrm{Hom}_{\mathcal{O}_q}({}_-,P_q(\mu)).
\end{displaymath}
The following result generalizes  \cite[Struktursatz]{So}.

\begin{thm}\label{thm4.3.2}
The functor $\Phi$ is an isomorphism of categories.
\end{thm}

\begin{proof}
By definition, $\Phi$ induces a bijection on objects. So we need only to check that it
induces a bijection on morphisms, that is that for any  $\lambda,\mu\in X$ the map
$\Phi_{\lambda,\mu}:\mathrm{Hom}_{\mathcal{O}_q}(P_q(\lambda),P_q(\mu))\to
\overline{\cC}(\mathrm{M}_{\lambda},\mathrm{M}_{\mu})$ is an isomorphism.
This is clear if both $P_q(\lambda)$ and $P_q(\mu)$ are injective.

By Proposition~\ref{prop4.3.1}, the injective envelope of $P_q(\mu)$ is projective.
Observe that, if $\varphi\in \mathrm{Hom}_{\mathcal{O}_q}(P_q(\lambda),P_q(\mu))$ is nonzero,
then the image of $\varphi$ contains a simple submodule $L$ in the socle of $P_q(\mu)$.
By Theorem~\ref{Thm3.9.1}, $L$ is a homomorphic image of some projective-injective
module $P$. By the projectivity of $P$, the surjection $f:P\rightarrow L$ lifts 
to a map $f':P\to P_q(\lambda)$ such that $f=\varphi\circ f'$. This implies that 
$\Phi_{\lambda,\mu}(\varphi)$ is nonzero and hence $\Phi_{\lambda,\mu}$ is injective.

To prove surjectivity let $\lambda,\mu\in X$ and $f\in\overline{\cC}(\mathrm{M}_{\lambda},\mathrm{M}_{\mu})$.
By Proposition~\ref{prop4.3.1}, there are exact sequences
\begin{displaymath}
0\rightarrow P_q(\lambda)\rightarrow X_1\rightarrow X_2\quad\text \quad
0\rightarrow P_q(\mu)\rightarrow Y_1\rightarrow Y_2
\end{displaymath}
in $\mathcal{O}_q$ such that $X_1$, $X_2$, $Y_1$ and $Y_2$ are projective-injective. 
Applying the covariant functor $\mathrm{Hom}_{\mathcal{O}_q}({}_-,{}_-)$ to these exact sequence
yields injective resolution for both $\mathrm{M}_{\lambda}$ and $\mathrm{M}_{\mu}$ in 
$\mathrm{mod}\text{-}\cC^{PI}$. The map $f$ admits lifts giving the following commutative diagram:
\begin{displaymath}
\xymatrix{ 
0\ar[rr]&& \mathrm{M}_{\lambda}\ar[rr]\ar[d]^f && 
\mathrm{Hom}_{\mathcal{O}_q}({}_-,X_1)\ar[rr]\ar[d]^{f'} && 
\mathrm{Hom}_{\mathcal{O}_q}({}_-,X_2)\ar[d]^{f''}\\
0\ar[rr]&& \mathrm{M}_{\mu}\ar[rr] && 
\mathrm{Hom}_{\mathcal{O}_q}({}_-,Y_1)\ar[rr]&& \mathrm{Hom}_{\mathcal{O}_q}({}_-,Y_2)\\
}
\end{displaymath}
As $X_1$, $X_2$, $Y_1$ and $Y_2$ are projective-injective, the right hand square of the latter diagram
is the image of the right hand square of some commutative diagram of the form
\begin{displaymath}
\xymatrix{ 
0\ar[rr]&& P_q(\lambda)\ar[rr]\ar@{.>}[d]^{\varphi} && X_1\ar[rr]\ar[d] && X_2\ar[d]\\
0\ar[rr]&& P_q(\mu)\ar[rr] && Y_1\ar[rr]&& Y_2\\
}
\end{displaymath}
As both rows are exact, the commutative right hand square of the latter diagram induces
a unique $\varphi:P_q(\lambda)\to P_q(\mu)$ making the digram commutative and we have 
$\Phi_{\lambda,\mu}(\varphi)=f$. This proves surjectivity and completes the proof.
\end{proof}

Statements similar to Proposition~\ref{prop4.3.1} and Theorem~\ref{thm4.3.2} appear frequently and play 
important role in Lie-theoretic context (sometimes in disguise), see e.g. \cite[Theorem~10.1]{St}, 
\cite[Theorem~3.9]{St2} and \cite[Corollary~2]{Ma}. Our proof above follows the approach of \cite{KSX}.
Making a parallel with the results of \cite{MM}, we propose the following conjecture:

\begin{conj}
The category $\cC^{PI}$ is symmetric, i.e. the bimodules $\cC^{PI}({}_-,{}_-)^*$ and
$\cC^{PI}({}_-,{}_-)$ are isomorphic.
\end{conj}

In \cite{MM} it is shown that some similar categories associated to certain Lie superalgebras are symmetric
using a description of the Serre functor for the corresponding category $\mathcal{O}$ via Harish-Chandra
bimodules. A similar approach for $\cC^{PI}$ would require development of the theory of Harish-Chandra
bimodules in the quantum case.

\subsection{Tilting modules in $\mathcal O_q$}\label{s4.4}

A module $M \in \mathcal O_q$ is called {\em tilting} if $M$ has both a Verma filtration and a dual 
Verma filtration. In $\mathcal O_{int}$ there exists, for each $\lambda \in X$, a unique indecomposable tilting 
module $T_\C(\lambda)$ which has $\lambda$ as its unique highest weight. The same is true in $\mathcal O_q$:

\begin{thm}\label{Thm18}
For each $\lambda \in X$ there exists an indecomposable tilting module $T_q(\lambda)$ with $\lambda$ as its unique highest weight. Every indecomposable tilting
module in $\mathcal O_q$ is isomorphic to $T_q(\lambda)$ for some $\lambda \in X$.
\end{thm}

There are various ways to prove this (compare e.g. with \cite{So3}), we choose the one which we think is the shortest.

\begin{proof}
The functor $G: \mathcal O_{int} \to \mathcal O_q^{spec}$ clearly takes tilting modules in $\mathcal O_{int}$ to tilting modules in $\mathcal O_q$, see Proposition~\ref{prop4}.
Hence for $\mu \in X$ we set $T_q(l\mu + (l-1)\rho) = T_\C(\mu)^{[l]} \otimes St_l$.

For general $\lambda \in X$ we set $\mu = \lambda^1 -\rho$ and consider $T = T_\C(\mu)^{[l]} \otimes St_l \otimes L_q(\lambda^0 +\rho)$. Then $T$ is a tilting module and its
highest weight is $\lambda$ occurring with multiplicity $1$. So we set $T_q(\lambda)$ equal to the unique indecomposable summand of $T$ which has a non-zero $\lambda$-weight space.

This gives the existence of $T_q(\lambda)$. The second statement is then seen by standard arguments, see \cite[Theorem 11.2]{Hu}.
\end{proof}

For $N\in\mathcal{O}_q$ we denote by $\mathrm{Tr}_{PI}(N)$ the {\em trace} in $N$ of all projective-injective
modules, that is the sum of the images of all homomorphisms from $M$ to $N$, where $M$ is projective-injective.
Note that for every finite dimensional $V\in \mathcal{O}_q$ the functor $V\otimes{}_-$ preserves the category 
of projective-injective modules. This implies that for any $N\in\mathcal{O}_q$ we have 
$\mathrm{Tr}_{PI}(V\otimes N)\cong V\otimes \mathrm{Tr}_{PI}(N)$.
Titling modules in $\mathcal{O}_q$ can be alternatively  described as follows:

\begin{thm}\label{thm4.4.1}
\begin{enumerate}[$($i$)$]
\item\label{thm4.4.1-1} For every $\lambda\in X$ the module $\mathrm{Tr}_{PI}(P_q(\lambda))$ is
an indecomposable tilting module.
\item\label{thm4.4.1-2} Every indecomposable tilting module is isomorphic to $\mathrm{Tr}_{PI}(P_q(\lambda))$
for some $\lambda\in X$.
\item\label{thm4.4.1-3} \rm{(Ringel self-duality)} For every $\lambda,\mu\in X$ we have 
\begin{displaymath}
\mathrm{Hom}_{\mathcal{O}_q} (P_q(\lambda),P_q(\mu))\cong
\mathrm{Hom}_{\mathcal{O}_q} (\mathrm{Tr}_{PI}(P_q(\lambda)),\mathrm{Tr}_{PI}(P_q(\mu))).
\end{displaymath}
\end{enumerate}
\end{thm}

In the classical case Ringel self-duality is due to Soergel, see \cite{So2}. 

\begin{proof}
This is well-known for $\mathcal{O}_{\mathbb{C}}$, see e.g. \cite{So2,FKM}. Hence
Theorem~\ref{thm9} implies the claim for $\mathcal{O}_q^{\mathrm{spec}}$.
Using translation and Theorem~\ref{Thm13} we obtain that $\mathrm{Tr}_{PI}(P_q(\lambda))$
is a tilting module for every $\lambda\in X$.

For every $\lambda,\mu\in X$ from Theorem~\ref{thm4.3.2} it follows that the restriction map
\begin{displaymath}
\mathrm{Hom}_{\mathcal{O}_q} (P_q(\lambda),P_q(\mu))\rightarrow
\mathrm{Hom}_{\mathcal{O}_q} (\mathrm{Tr}_{PI}(P_q(\lambda)),\mathrm{Tr}_{PI}(P_q(\mu)))
\end{displaymath}
is bijective. This proves \eqref{thm4.4.1-3} and implies that every $\mathrm{Tr}_{PI}(P_q(\lambda))$
is indecomposable, proving \eqref{thm4.4.1-1}. Claim \eqref{thm4.4.1-2} follows from the fact that 
every tilting module occurs as a direct summand of a simple tilting module from $\mathcal{O}_q^{\mathrm{spec}}$
tensored with a finite dimensional module.
\end{proof}

Theorem~\ref{thm4.4.1}\eqref{thm4.4.1-1}
combined with Theorem~\ref{Thm13} implies a tensor product formula for indecomposable tilting modules 
similar to Theorem~\ref{Thm13} and Theorem~\ref{Thm14}. Namely, 
let $\lambda^0 \in X_l$ and write $\tilde \lambda^0 = l\rho + w_0\cdot \lambda^0$. 
 
\begin{cor}\label{cor4.4.2}
For each $\lambda \in X$ we have $T_q(\lambda) \simeq T_\C(\lambda^1-\rho)^{[l]}~\otimes~Q(\tilde \lambda^0)$.
\end{cor}

\section{Characters and Kazhdan-Lusztig data}\label{s6}

\subsection{Character formulas}\label{s6.1}
Consider the group ring $\Z[X]$ in which we denote the basis element corresponding 
to $\lambda \in X$ by $e^\lambda$. The multiplication is 
then determined by $e^\lambda e^\mu = e^{\lambda + \mu}$. 

We extend this ring by defining its ``completion'' $\widehat{\Z[X]}$ 
to consist of all expressions $\sum_{\lambda} c_\lambda e^\lambda$ where $c_\lambda \in \Z$ for all $\lambda$ and
there exist $\lambda_1, \dots , \lambda_r \in X$ such that $c_\lambda = 0$ unless $\lambda \leq \lambda_i$ for some $i$ (here $\leq$ is the usual order on $X$). Alternatively, this is the set of $\Z$-valued functions
on $X$ whose support is contained in a finite union of subsets of the form 
$X_{\leq \mu} = \{\lambda \in X \mid \lambda \leq \mu \}$. Clearly, the multiplication on $\Z[X]$ extends to
$\widehat{\Z[X]}$.

If $f = \sum a_\lambda e^\lambda \in \widehat{\Z[X]}$, we set $f^{[l]} = \sum a_\lambda e^{l\lambda}$.
If $M \in \mathcal O_{\mathrm{int}}$ or $M \in \mathcal O_q$, we set 
$\ch M = \sum_\mu (\dim M_\mu) e^\mu \in \widehat{\Z[X]}$ and call this the character of $M$. Then for $M \in \mathcal O_{\mathrm{int}}$ we get
$\ch (M^{[l]}) = (\ch M)^{[l]}$.

Using the notation from
Section~\ref{s3.4} for $M \in \mathcal O_{\mathrm{int}}$ we have 
\begin{equation}\label{eq4}
\ch M = \sum_\mu [M:L_\C(\mu)] \ch L_\C(\mu) ,
\end{equation} 
and similarly for $M\in \mathcal O_q$ we have
\begin{equation}\label{eq5}
\ch M = \sum_\mu [M:L_q(\mu)] \ch L_q(\mu).
\end{equation}
These sums are finite, cf. Corollary~\ref{prop6}.
If we take $M = \Delta_\C(\lambda)$, then the sum in \eqref{eq4} has a unique highest term, namely $ 1 \cdot \ch L_\C(\lambda)$. We can therefore ``invert'' these equations and
obtain
\begin{equation}\label{eq6}
\ch L_\C(\lambda) = \sum_\mu p^\C_{\mu, \lambda}  \ch \Delta_\C(\mu) 
\end{equation}
for some unique $p^\C_{\mu, \lambda} \in \Z$. Similarly, we get 
 \begin{equation}\label{eq7}
\ch L_q(\lambda) = \sum_\mu p^q_{\mu, \lambda}  \ch \Delta_q(\mu) 
\end{equation}
for some unique $p^q_{\mu, \lambda} \in \Z$. 

Note that whereas the sum in \eqref{eq6} is finite for all $\lambda \in X$ (we have 
$p^\C_{\mu, \lambda} = 0$ unless $\mu \in W\cdot \lambda$), this is not so in \eqref{eq7} (as blocks of 
$\mathcal O_q$ could have infinitely many simples). 
For instance in the $\mathfrak{sl}_2$-case we have
\begin{equation}\label{eq8}
\ch L_q(-2) = \Delta_q(-2) + \sum_{m \leq -1} (\ch \Delta_q(2ml) - \ch \Delta_q(2ml -2)). 
\end{equation}
Similarly, we may consider the characters of (finite dimensional) $u_qB_q$-modules. Here we obtain the analogous formulas
\begin{equation}\label{eq9}
\ch M = \sum_\mu [M:\tilde L_q(\mu)] \ch \tilde L_q(\mu).
\end{equation}
and
 \begin{equation}\label{eq10}
\ch \tilde L_q(\lambda) = \sum_\mu \tilde p^q_{\mu, \lambda}  \ch \tilde \Delta_q(\mu) 
\end{equation}
for some unique $\tilde p^q_{\mu, \lambda} \in \Z$.
Again, \eqref{eq9} clearly involves only finite sums for any finite dimensional $M$ (and is in fact a formula in $\Z[X]$) whereas the sum in \eqref{eq10} may well be infinite.

Finally, we observe the following obvious identities
\begin{equation}\label{eq12}
(\ch \Delta_C(\lambda)) e^\mu = \ch \Delta_\C(\lambda + \mu);
\end{equation}
\begin{equation}\label{eq13}
(\ch \Delta_q(\lambda)) e^\mu = \ch \Delta_q(\lambda + \mu);
\end{equation}
and
\begin{equation}\label{eq14}
(\ch \tilde \Delta_q(\lambda)) e^\mu = \ch \tilde \Delta_q(\lambda + \mu)
\end{equation}
valid for all $\lambda, \mu \in X$.

\subsection{Characters of simple modules in $\mathcal O_q$}

Using the terminology from Section~\ref{s6.1} we have:

\begin{thm}\label{thm22}
For all $\lambda \in X$ we have the following:
\begin{enumerate}[$($i$)$] 
\item\label{thm22.1}  $\ch L_q(\lambda) = \sum_{\nu, \eta} p^\C_{\nu, \lambda^1} \tilde p^q_{\eta, \lambda^0} \ch \Delta_q(l\nu + \eta)$;
\item\label{thm22.2} $p^q_{\mu, \lambda} = \sum_{l\nu + \eta = \mu}  p^\C_{\nu, \lambda^1}\tilde p^q_{\eta, \lambda^0} =
 \sum_{w \in W^{\lambda^1}} p^\C_{w\cdot \lambda^1, \lambda^1}\tilde p^q_{\mu - lw\cdot \lambda^1, \lambda^0}$, where
 $W^{\lambda^1}$ denotes the set of shortest coset representatives in $W/\mathrm{Stab}_{W\cdot}(\lambda^1)$.
\end{enumerate}
\end{thm}
 
 \begin{proof}
 By Proposition~\ref{prop3} combined with \eqref{eq6} and \eqref{eq10} we find
 \begin{multline*}
 \ch L_q(\lambda) = (\ch L_\C(\lambda^1)^{[l]} \ch L_q(\lambda^0) =\\= \sum_\nu p^\C_{\nu, \lambda^1} (\ch \Delta_\C(\nu))^{[l]} \sum_\eta \tilde p^q_{\eta, \lambda^0}\ch  \tilde
 \Delta_q(\eta) = \\ 
 = \sum_\mu \left(\sum_{l\nu + \eta = \mu} p^\C_{\nu, \lambda^1} \tilde p^q_{\eta, \lambda^0}\right) (\ch \Delta_\C(\nu))^{[l]} \ch \tilde \Delta_q(\eta).
 \end{multline*} 
 So, to establish \eqref{thm22.1} we should only check that 
 \begin{equation}\label{eq14a}
 (\ch \Delta_\C(\nu))^{[l]} \ch \tilde \Delta_q(\eta) = \ch \Delta_q(l \nu + \eta).
 \end{equation}
 However, by \eqref{eq14} we have  $\ch \tilde \Delta_q(\eta) = 
 (\ch St_l) e^{\eta-(l-1)\rho}$ (because $St_l = \tilde \Delta_q((l-1)\rho)$). 
 Hence using Proposition~\ref{prop4} we find
 \begin{multline*}
 (\ch \Delta_\C(\nu))^{[l]} \ch \tilde \Delta_q(\eta) = 
 (\ch \Delta_\C(\nu))^{[l]} (\ch St_l) e^{\eta - (l-1)\rho} =\\
 = (\ch \Delta_q(l \nu + (l-1)\rho)) e^{\eta - (l-1)\rho} =
 \ch \Delta_q(l\nu + \eta). 
 \end{multline*}
Here we have used \eqref{eq13} for the last equality.
 
 The first equality in (ii) is immediate from (i) and the second comes from the fact that $p^\C_{\nu, \lambda^1} = 0$ unless $\nu \in W \cdot \lambda^1$.
 \end{proof}
 
 \subsection{Characters of indecomposable tilting modules in $\mathcal O_q$}
 By Corollary~\ref{cor4.4.2} we get for any $\lambda \in X$
 \[
 \ch T_q(\lambda) = \sum_{\nu, \eta} (T_\C(\lambda^1-\rho): \Delta_\C(\nu))(Q_q(\tilde \lambda^0): \tilde \Delta_q(\eta)) \ch (\Delta_\C(\nu)^{[l]}) \ch \tilde \Delta_q(\eta).
 \]
 Applying \eqref{eq14a} in this formula we get
 
 \begin{thm}\label{thm22a}
 For all $\lambda, \mu \in X$ we have
 $$ (T_q(\lambda):\Delta_q(\mu)) = \sum_{\nu, \eta, l\nu + \eta = \mu} (T_\C(\lambda^1-\rho):\Delta_C(\nu))(Q_q(\tilde \lambda^0):\tilde \Delta_q(\eta)).$$
 \end{thm}
 
 \subsection{Kazhdan-Lusztig theory for $\mathcal O_q$}
 
 Fix an antidominant weight $\lambda \in X$. For each $\mu \in W\cdot \lambda$ we pick $w \in W$ minimal such that $w\cdot \lambda = \mu$. Then the Kazhdan-Lusztig conjecture
 \cite{KL} proved independently by Beilinson and Bernstein \cite{BB}, and by Brylinski and Kashiwara \cite{BK} says (for each such minimal $y, w \in W$)
 \begin{equation}\label{eq15}
p^\C_{y\cdot \lambda, w \cdot \lambda} = (-1)^{l(yw)} P_{y, w} (1).
\end{equation}
Here $P_{y, w}$ is the Kazhdan-Lusztig polynomial associated to $y, w$, see \cite{KL}.

As discussed in the introduction the analogous Lusztig conjecture for finite dimensional simple modules in $\mathcal F_q$ has also been settled.

Let $W_l$ be the affine Weyl group (of dual Langlands type).
Set 
$$A^-_l = \{\lambda \in X \mid -l < \langle \lambda + \rho, \alpha^{\vee}\rangle < 0 
\text { for all positive roots } \alpha \}.$$ 
This is the top antidominant alcove. Fix
$\lambda \in \bar A^-_l$ and choose for each $\mu \in W_l \cdot \lambda$ a minimal $x \in W_l$ such that $\mu = x\cdot \lambda$. Then in analogy with
\eqref{eq15} for all such minimal $z, x \in W_l$ for which $z\cdot \lambda,x\cdot \lambda \in {X^+}$ we have
\begin{equation}\label{eq16}
p^q_{z\cdot \lambda, x \cdot \lambda} = (-1)^{l(zx)} P_{z, x} (1).
\end{equation}
Here $P_{z,x}$ is again the Kazhdan-Lusztig polynomial associated to the pair $(z,x)$ in the affine Weyl group
$W_l$. Note that this gives us only some of the coefficients, the remaining coefficients 
(for $x\cdot \lambda\in {X^+}$ fixed) can be obtained using
the Weyl character formula for quantum Weyl modules.

Combining the above two formulas we get:

\begin{cor}\label{cor23}
Let $\lambda \in X$ and suppose $w\in W$, resp. $x\in W_l$ is minimal such that $w^{-1} \cdot \lambda^1$ is antidominant, resp. $x^{-1} \cdot \lambda^0 \in 
\bar A^-_l$. Then the character of $L_q(\lambda)$ equals
\begin{multline*}
\sum_{r\in W}\,\,\,\sum_{y\in W^{w^{-1} \cdot \lambda^1}}\,\,\,
\sum_{\stackrel{z\in W_l^{x^{-1}\cdot \lambda^0}}{zx^{-1}\cdot \lambda^0\in{X^+}}} 
(-1)^{l(yw)+l(zx)+l(r)}P_{y,w}(1) P_{z,x}(1)\cdot\\\cdot 
\ch \Delta_{\mathbb{C}}(yw^{-1}\cdot \lambda^1)^{[l]}\ch \Delta_q\big(rzx^{-1}\cdot \lambda^0) 
\end{multline*}
\end{cor}

Similarly, the result in Theorem~\ref{thm22a} leads to the following expression for the characters of indecomposable tilting modules in $\mathcal O_q$. The formula involves the
``inverse'' Kazhdan-Lusztig polynomials $Q_{x, y}$, i.e. the polynomials determined by the equations
\begin{displaymath}
\sum_z(-1)^{l(z) - l(y)} P_{y,z} Q_{z,w} = \delta_{y,w}.
\end{displaymath}
\begin{cor}\label{cor23a}
Let $\lambda \in X$. We assume that $\lambda^1-\rho$ is regular i.e. belongs to the interior of a chamber so that there is a unique $w\in W$ with 
$w^{-1} \cdot (\lambda^1-\rho)$ antidominant. Likewise we assume that $\lambda^0$ is $l$-regular so that there is a unique 
$x \in W_l$ with $x^{-1} \cdot \lambda^0 \in A^-_l$. Then
\[
(T_q(\lambda): \Delta_q(\mu)) = \sum_{y,z} P_{y,w}(1) Q_{z,x}(1)
\]
where the sum runs over those $y \in W$, $z \in W_l$ for which $\mu = lyw^{-1}\lambda^1 + zx^{-1}\cdot \lambda^0$.
\end{cor}

\begin{proof}
According to \cite[Conjecture 7.1]{So1} (proved in \cite{So2}) we have $(T_\C(\lambda^1 -\rho): \Delta_\C(yw^{-1} \cdot(\lambda^1-\rho)) = P_{y,w}(1)$ and 
$(Q_q(\tilde \lambda^0):\tilde \Delta_q(zx^{-1} \cdot \lambda^0 +l\rho) = Q_{z,x}(1)$.
\end{proof}

\begin{remark}
Using \cite[Remark 7.2.2]{So1} it is possible to generalize this corollary to include weights $\lambda$ without the stated regularity assumptions.
\end{remark}

\begin{exm}
Consider the simplest possible case where the Lie algebra  $\mathfrak g$ is  $\mathfrak{sl}_2$. Then $X = \Z$ with $X^+ = \Z_{\geq 0}$. Suppose $q$ is a complex root
of unity of odd order $l>1$. Then we have the following description of the Verma modules and the indecomposable tilting modules in $\mathcal O_q$:

\begin{enumerate}
\item [Case 1.] Suppose $\lambda \in \Z$ is $l$-singular, i.e. $\lambda \equiv -1$ (mod $l$). Then we are in the special block of $\mathcal O_q$ and 
just as in the classical case we have (cf. Theorem 3.11  $\Delta_q(\lambda) = T_q(\lambda) = L_q(\lambda)$  when $\lambda < 0$. On the other hand, for $\lambda \geq 0$ we have 
the following two short exact sequences
\[
 0 \to L_q(-\lambda -2) \to \Delta_q(\lambda) \to L_q(\lambda) \to 0,
\]
and
\[
0 \to \Delta_q(\lambda) \to T_q(\lambda) \to \Delta_q(-\lambda -2) \to 0.
\]

\item[Case 2.] Suppose $\lambda$ is $l$-regular and write $\lambda = \lambda^0 + l \lambda^1$ with $0 \leq \lambda^0 < l-1$. If $\lambda <0$ we have 
the following two exact sequences
\[
0 \to L_q(-\lambda^0 -2 +l\lambda^1) \to \Delta_q(\lambda) \to L_q(\lambda) \to 0,
\]
and
\[
0 \to \Delta_q(\lambda) \to T_q(\lambda) \to \Delta_q(-\lambda^0 -2 + l \lambda^1) \to 0.
\] 
On the other hand, if $\lambda \geq 0$ then  $
\Delta_q(\lambda)$ has the following four composition factors: $L_q(\lambda), L_q(-\lambda^0 -2 + l\lambda^1), L_q(-\lambda -2)$, 
and $L_q(\lambda^0 +l(-\lambda^1-2))$. Note that if
$\lambda^1 = 0$ the second and the third of these factors coincide and $\Delta_q(\lambda)$ has in this case only three composition factors.
The tilting module $T_q(\lambda)$ similarly has the following four Verma filtration factors: $\Delta_q(\lambda), \Delta_q(-\lambda^0-2 + l \lambda^1), \Delta_q(-\lambda
-2)$, and $\Delta_q(\lambda^0 - l\lambda^1)$. If $\lambda^1 = 0$ then the first and the last factors are identical and so are the second and third, so that in this case 
there are only two factors.
\end{enumerate}

\end{exm}

Although this example is maybe too simple to catch the full flavor of the behavior of Verma modules and tilting modules it does illustrate some important features: In the most
singular case (the special block) the modules are as in the classical case (with the Kazhdan-Lusztig polynomials for the finite ordinary Weyl group governing the combinatorics). At
the other extreme when the weight $\lambda$ in question is $l$-regular and far from the walls of the Weyl chamber $C$ containing it the Verma module $\Delta_q(\lambda)$, respectively the tilting module $T_q(\lambda)$
has composition factors, respectively Verma filtration factors, belonging to "clusters" in all the Weyl chambers below $C$ (in the Bruhat order).
These clusters look like the alcove patterns in \cite{Lu80}. When $\lambda$ is close to some walls there are (rather complicated) cancellations among these clusters. Moreover,
there are several degrees of $l$-singularity (depending on the facette for the affine Weyl group to which $\lambda$ belongs) which also influence the patterns. This is all
encoded in Corollaries \ref{cor23} and \ref{cor23a}.

\section{The generic case}\label{s5}
 
\subsection{The category $\mathcal O_v$}
 
We define the category $\mathcal O_v$ to be the full subcategory of the category of $U_v$-modules consisting of 
those modules which satisfy the analogues of 
\eqref{cond1}--\eqref{cond3} in Section \ref{s3.1}.
 
Among the objects in $\mathcal O_v$ we have the generic Verma modules $\Delta_v(\lambda)$, $\lambda \in X$, defined 
in the usual way. They have unique simple quotients $L_v(\lambda)$ and these are up to isomorphism a complete set of simple modules in $\mathcal O_v$.
 
 The category $\mathcal O_v$ has properties completely analogous to $\mathcal O_{int}$, see e.g. \cite[Chapters 9-10]{C-P}. In particular all modules in $\mathcal O_v$ have
 finite length, the Verma module $\Delta_v(\lambda)$ has composition factors $L_v(\mu)$ with $\mu \in W \cdot \lambda$, and in fact $\mathcal O_v$ splits into blocks
 \begin{displaymath}
 \mathcal O_v = \oplus_{\lambda} \mathcal O_v^\lambda
 \end{displaymath}
 where the block $\mathcal O_v^\lambda$ consists of those modules from $\mathcal O_v$ whose composition factors have highest weights in $W \cdot \lambda$, and where 
 the sum runs over the set of all $\lambda$ for which $\lambda + \rho$ are dominant.
 
In analogy to $\mathcal O_q$ we have a duality ${{}_-}^\star$ on $\mathcal O_v$ which fixes simple modules. The dual Verma module $\Delta_v^\star(\lambda)$ is denoted
 $\nabla_v(\lambda)$.
 
\subsection{$A$-lattices}
 
Clearly the Verma module $\Delta_v(\lambda)$ has an $A$-lattice, namely the Verma module for $U_A$ defined by 
 \begin{displaymath}
 \Delta_A(\lambda) = U_A \otimes _{B_A} A_\lambda.
 \end{displaymath}
 Here $B_A$ is the Borel subalgebra of $U_A$ defined in analogy with $B_q$ and $A_\lambda$ denotes the free rank one $A$-module with $B_A$-action given by
 the analogue over $A$ of the character $\chi_\lambda$ from Section \ref{s2.4}.
 
 Similarly, $\nabla_v(\lambda)$ has an $A$-lattice $\nabla_A(\lambda)$ defined as the $A$-dual of $\Delta_A(\lambda)$ (with the appropriate $U_A$-structure).
 
 Note that $\Hom_{U_A}(\Delta_A (\lambda), \nabla_A(\lambda)) \simeq A$. We let $c_\lambda$ denote a generator of this module and set $K_A(\lambda)$, respectively $L_A(\lambda)$, respectively $C_A(\lambda)$, 
 equal to the kernel, respectively the image, respectively the cokernel, of $c_\lambda$. Then we get the following two short exact sequences in $\mathcal O_v$:
 \[
 0 \to K_A(\lambda) \to \Delta_A(\lambda) \to L_A(\lambda) \to  0,
 \]
 and
 \[
 0 \to L_A(\lambda) \to \nabla_A(\lambda) \to C_A(\lambda) \to  0.
 \]
 Tensoring by the fraction field $\Q(v)$ of $A$ we see that $L_A(\lambda) \otimes _A \Q(v) \simeq L_v(\lambda)$ because $L_v(\lambda)$ is the image of
 $c_\lambda \otimes 1 : \Delta_v(\lambda) \to \nabla_v(\lambda)$. On the other hand, if we specialize to a root of unity $q \in \C$ (i.e. apply $- \otimes_A \C_q$ with 
 $\C_q $ denoting $\C$ made into an $A$-module by mapping $v$ to $q$) then we obtain the following two exact sequences in $\mathcal O_v$
 \[
  K_A(\lambda)\otimes _A \C_q \to \Delta_q(\lambda) \to L_A(\lambda) \otimes _A \C_q \to  0,
 \]
 and
 \[
 0 \to \Tor_1^A(C_A(\lambda), \C_q) \to L_A(\lambda) \otimes _A \C_q \to \nabla_q(\lambda) \to C_A(\lambda)\otimes_A \C_q \to  0.
 \]  
 As $L_A(\lambda) \otimes _A \C_q$ is a non-zero  quotient of $\Delta_q(\lambda)$ it has $L_q(\lambda)$ as a quotient but it may be bigger. 
 
 \begin{prop}\label{prop19}
 Let $\lambda, \mu \in X$ be fixed. Then 
 \begin{displaymath}
 \dim_\C L_q(\lambda)_\mu \leq \dim_{\Q(v)} L_v(\lambda)_\mu
 \end{displaymath}
 for all $l$. Equality holds if $l \gg 0$.
 \end{prop}
 
 \begin{proof}
 The inequality follows from the above considerations. They also show that we have equality if and only if $ \Tor_1^A(C_A(\lambda)_\mu, \C_q) = 0$. But $C_A(\lambda)_\mu$ is
 a finitely generated $A$-module so this $\Tor$ vanishes for all but at most finitely many $q$.
 \end{proof}

\subsection{Generic multiplicities}

From the above discussion we deduce the following result, comparing combinatorics of $\mathcal{O}_v$
with that of $\mathcal{O}$ (see \cite{HK} for some other results in the same spirit).

\begin{thm}\label{Thm20}
Let $\lambda \in X$. Then we have
\begin{displaymath}
\ch L_v(\lambda) = \ch L_\C(\lambda).
\end{displaymath}
\end{thm}

\begin{proof}

We want to prove that $\dim_{\Q(v)} L_v(\lambda)_\mu = \dim_{\C}
L_{\C}(\lambda)_{\mu}$ for all $\mu \in X$. So let us fix $\mu$. By
Proposition 6.1 we have $\dim_{\Q(v)} L_v(\lambda)_\mu = \dim_{\C} L_q
(\lambda)_{\mu}$ for $l$ large. Moreover, by Theorem 5.1 (i) we have
$$\dim L_q(\lambda)_\mu = \sum_{\nu, \eta} p^\C_{\nu, \lambda^1} \tilde
p^q_{\eta, \lambda^0} \dim \Delta_q(l\nu + \eta)_\mu.$$
Note that the pair $(\nu, \eta)$ on the right hand side only contributes a
non-zero term if $\nu \leq \lambda^1$ and $\eta \leq \lambda^0$. However,
if $\nu < \lambda^1$ we have $\Delta_q(l\nu + \eta)_\mu =0$ for all $\eta
\leq \lambda^0$ when $l$ is large enough, i.e. in that case the right hand side
reduces to 
$$\sum_{\nu \leq \lambda^0} \tilde p^q_{\eta, \lambda^0} \dim
\Delta_q(l\lambda^1 + \eta)_\mu.$$
We have $\tilde p^q_{\eta, \lambda^0} = \tilde p^q_{\eta + l\lambda^1,\lambda} $ 
so that we are done if we check that $\tilde p^q_{\eta,
\lambda} = p^{\C}_{\eta, \lambda} $ for all $\eta$ with $\mu \leq
\eta \leq \lambda$ (and for $l$ large).

The right hand side is given by \eqref{eq15}. Let us consider the left hand
side. We choose $l$ large enough so that all the $\eta$'s we consider
belong to the alcoves having $-\rho$ in their closures. The minimal $w \in
W_l$ satisfying $w^{-1}\cdot \lambda \in \bar A_l^l$ then belongs to $W$.
The same is true for the minimal $y$ such that $\eta = yw^{-1} \cdot
\lambda$. Hence the quantum Kazhdan-Lusztig conjecture 
(cf. e.g. \cite[Theorem~6.1]{So1}) says
$$ \tilde p^q_{yw^{-1}\cdot \lambda, \lambda} = (-1)^{l(yw)} P_{y,w}(1)$$
and we are done.
\end{proof}

From Theorem~\ref{Thm20} for all $\lambda,\mu \in X$ we have 
\begin{displaymath}
[\Delta_v(\lambda):L_v(\mu)] = 
[\Delta_\C(\lambda):L_\C(\mu)].
\end{displaymath}
Furthermore, we also have

\begin{cor}\label{corn21}
The result in Theorem~\ref{Thm20} holds not only for an indeterminate $v$ but also for 
any specialization $v \mapsto q$  where $q\in \C$ is transcendental.
\end{cor}

\begin{remark}\label{remn21}
The results of this section also follow from the fact that, after extending the scalars to $\mathbb{C}$, 
for any transcendental $q\in\mathbb{C}\setminus\{0\}$ the categories $\mathcal{O}_q$ 
and $\mathcal{O}_{\mathrm{int}}$ are equivalent (each block of such category can be realized as 
modules over an associative algebra with relations defined over $\mathbb{Q}$).   
As was pointed out to us by D.~Kazhdan, such an equivalence was established 
(in the simply laced case) by M. Finkelberg in his thesis \cite{Fi}.  
G.~Lusztig made us aware of the paper \cite{EK} where such an equivalence is established 
more generally for  symmetrizable Kac-Moody algebras, see \cite[Theorem 4.2]{EK}. 
In our case this result gives that if $v$ is an indeterminate then the category $\mathcal O_v$ is equivalent to the 
integral block in the category $\mathcal O$ for $\mathfrak g \otimes_\Q K$ where $K$ is the fraction field of $\Q[[v]]$.

The methods used in both 
\cite{Fi} and \cite{EK} are completely different from our approach.
\end{remark}

\section{A parallel with Lie superalgebras}\label{s9}

There are several similarities between general properties of $\mathcal{O}_q$ and those of the category $\mathcal{O}$
for finite dimensional Lie superalgebras. These similarities played an important role in forming our intuition for the results of the present paper and in this section we try to make them more precise, following a suggestion of the 
referee and the editor. We refer the reader e.g. to  \cite{Mu} for more details on Lie superalgebras and 
their modules. For a Lie superalgebra $\mathfrak{a}$ we denote by $U(\mathfrak{a})$ the corresponding 
enveloping algebra.

\subsection{Super setup}\label{s9.1}

Let $\mathfrak{g}=\mathfrak{g}_{\overline{0}}\oplus \mathfrak{g}_{\overline{1}}$ be a Lie superalgebra over 
$\mathbb{C}$. We assume that $\mathfrak{g}_{\overline{0}}$ is a finite dimensional reductive Lie algebra and
$\mathfrak{g}_{\overline{1}}$ is a semi-simple $\mathfrak{g}_{\overline{0}}$-mo\-dule. We denote by 
$\mathfrak{g}\text{-}\mathrm{smod}$ the abelian category of $\mathfrak{g}$-supermodules (where morphisms are
homogeneous $\mathfrak{g}$-homomorphisms of degree $0$). Fix some triangular decomposition
$\mathfrak{g}=\mathfrak{n}^-\oplus \mathfrak{h}\oplus \mathfrak{n}^+$ with the induced triangular decomposition
$\mathfrak{g}_{\overline{0}}=\mathfrak{n}^-_{\overline{0}}\oplus \mathfrak{h}_{\overline{0}}\oplus 
\mathfrak{n}^+_{\overline{0}}$ for $\mathfrak{g}_{\overline{0}}$. Our two basic examples are: the general linear
Lie superalgebra $\mathfrak{gl}(m\vert n)$ and the queer Lie superalgebra $\mathfrak{q}_n$. 

The  superalgebra $\mathfrak{gl}(m\vert n)$ consists of $(n+m)\times(n+m)$ matrices naturally divided into
$n\times n$, $n\times m$, $m\times n$ and $m\times m$ blocks. The operation is the usual super-commutator of
matrices. The diagonal blocks form the even part while the off-diagonal blocks form the odd part. The standard 
triangular decomposition corresponds to taking lower triangular, diagonal and upper triangular matrices. Note 
that the Cartan subalgebra $\mathfrak{h}$ of diagonal matrices is purely even.

The  superalgebra $\mathfrak{q}_n$ consists of $2n\times 2n$ matrices  of the form
\begin{displaymath}
\left(\begin{array}{c|c}A&B\\\hline B&A\end{array}\right)
\end{displaymath}
with respect to the usual super-commutator  of matrices. The even part corresponds to $B=0$ while the odd
part corresponds to $A=0$. The triangular decomposition is induced by the standard triangular decomposition 
for $A$ and $B$ (simultaneously). In this example the Cartan subalgebra $\mathfrak{h}$ has a nonzero
odd component, in particular, it is not commutative.

\subsection{Category $\mathcal{O}$}\label{s9.2}

To avoid technicalities and complicated notation, we will describe the situation for the classical category
$\mathcal{O}$ for $\mathfrak{g}$. All properties transfer mutatis mutandis to the parabolic versions of
$\mathcal{O}$.

For a Lie superalgebra with triangular decomposition as above we can consider the corresponding category
$\mathcal{O}$ defined as the full subcategory of $\mathfrak{g}\text{-}\mathrm{smod}$ containing all objects 
$M$ with the following properties:
\begin{itemize}
\item [i)]$M$ is finitely  generated;
\item [ii)]the action of $\mathfrak{h}_{\overline{0}}$ on $M$ is diagonalizable;
\item [iii)] the action of $U(\mathfrak{n}^+)$ on $M$ is locally finite.
\end{itemize}
Let $V$  be a simple $\mathfrak{h}_0$-diagonalizable $\mathfrak{h}$-supermodule. Set $\mathfrak{n}^+\cdot V=0$
and define the {\em Verma} or {\em proper standard} supermodule $\overline{\Delta}(V)$ as usual via
\begin{displaymath}
\overline{\Delta}(V):=U(\mathfrak{g})\otimes_{U(\mathfrak{h}\oplus \mathfrak{n}^+)} V.
\end{displaymath}
Standard arguments (see \cite{Di,Hu}) show that $\overline{\Delta}(V)$ has simple top, which we denote by $L(V)$.
Then the map $V\mapsto L(V)$ sets up a bijection from the set isomorphism classes of  simple 
$\mathfrak{h}_0$-diagonalizable $\mathfrak{h}$-supermodules to the set of isomorphism classes of simple objects 
in $\mathcal{O}$.

Let $\mathcal{O}_{\overline{0}}$ denote the category $\mathcal{O}$ for $\mathfrak{g}_{\overline{0}}$.
As $U(\mathfrak{g})$ is finite over $U(\mathfrak{g}_{\overline{0}})$ by the PBW theorem,
we have the usual induction functor
\begin{displaymath}
\mathrm{Ind}_{\mathfrak{g}_{\overline{0}}}^{\mathfrak{g}}:
\mathcal{O}_{\overline{0}}\to \mathcal{O}.
\end{displaymath}
Our first observation is the following (compare with Corollary~\ref{prop6}):

\begin{prop}\label{prop91}
\begin{enumerate}[$($a$)$]
\item\label{prop91.1} The restriction functor $\mathrm{Res}_{\mathfrak{g}_{\overline{0}}}^{\mathfrak{g}}$ 
maps $\mathcal{O}$ to $\mathcal{O}_{\overline{0}}$.
\item\label{prop91.2} Every object in $\mathcal{O}$ has finite length.
\end{enumerate}
\end{prop}

\begin{proof}
From the PBW Theorem it follows that 
$\mathrm{Res}_{\mathfrak{g}_{\overline{0}}}^{\mathfrak{g}}\, \overline{\Delta}(V)$ is in $\mathcal{O}_{\overline{0}}$,
which implies $\mathrm{Res}_{\mathfrak{g}_{\overline{0}}}^{\mathfrak{g}}\, L(V)\in \mathcal{O}_{\overline{0}}$.
By adjunction, a projective cover $P(V)$ for $L(V)$ is a direct summand of some module of the form
$\mathrm{Ind}_{\mathfrak{g}_{\overline{0}}}^{\mathfrak{g}}\, X$, where $X\in \mathcal{O}_{\overline{0}}$.

By the PBW Theorem, the composition $\mathrm{Res}_{\mathfrak{g}_{\overline{0}}}^{\mathfrak{g}}\circ
\mathrm{Ind}_{\mathfrak{g}_{\overline{0}}}^{\mathfrak{g}}$ is naturally isomorphic to tensoring with the 
finite dimensional $\mathfrak{g}_{\overline{0}}$-module $\bigwedge\mathfrak{g}_{\overline{1}}$.
This means that $P(V)$ is of finite length already as a $\mathfrak{g}_{\overline{0}}$-module, in particular,
it is of finite length as an object in $\mathcal{O}$. This implies claim \eqref{prop91.2} and claim
\eqref{prop91.1} follows by induction on the length of a module.
\end{proof}

Both $\mathrm{Res}_{\mathfrak{g}_{\overline{0}}}^{\mathfrak{g}}$ and 
$\mathrm{Ind}_{\mathfrak{g}_{\overline{0}}}^{\mathfrak{g}}$ are exact by the PBW Theorem (and even biadjoint up 
to parity shifts, see e.g. \cite{Go}). In particular, the two functors take projectives to projectives and injectives 
to injectives. The PBW Theorem also implies that they map modules with Verma filtrations to modules with
Verma filtrations. 

Usual arguments imply that $\mathcal{O}$ has a block decomposition into a direct sum of indecomposable 
subcategories. The first principal difference with $\mathcal{O}_{\overline{0}}$ is that blocks of 
$\mathcal{O}$ might contain infinitely many isomorphism classes of simple objects (however, always at most
countably many). Description of the linkage principle for simples in $\mathcal{O}$ is a very hard problem in 
full generality. However, it is known for several special cases, see e.g. \cite{Br1} for
$\mathfrak{gl}(m\vert n)$ or \cite{Br2} for $\mathfrak{q}_n$. 

\subsection{Standardly stratified structure}\label{s9.3}

Let $V$  be a simple $\mathfrak{h}_0$-dia\-gonal\-izable $\mathfrak{h}$-supermodule and
$\hat{V}$ be its projective cover in the category of $\mathfrak{h}_0$-diagonalizable $\mathfrak{h}$-supermodules.
Set $\mathfrak{n}^+\cdot \hat{V}=0$  and define the {\em standard} module $\Delta(V)$ as follows:
\begin{displaymath}
{\Delta}(V):=U(\mathfrak{g})\otimes_{U(\mathfrak{h}\oplus \mathfrak{n}^+)} \hat{V}. 
\end{displaymath} 
From the PBW Theorem it follows that every projective in $\mathcal{O}$ has a {\em standard filtration}, that is
a filtration with standard subquotients. Clearly, each ${\Delta}(V)$ has a filtration by Verma 
(proper standard) supermodules. Note that $\mathrm{Ind}_{\mathfrak{g}_{\overline{0}}}^{\mathfrak{g}}$ maps modules 
with Verma filtrations to modules with standard filtrations. 

Dually, one defines the costandard module ${\nabla}(V)$ and the proper costandard module
$\overline{\nabla}(V)$. Using adjunction and the fact that 
$\mathcal{O}_{\overline{0}}$ is a highest weight category, we get the following
ext-vanishing property:
\begin{displaymath}
\mathrm{Ext}_{\mathcal{O}}^i({\Delta}(V),\overline{\nabla}(V'))\cong
\begin{cases}
\mathbb{C}& \text{if }V\cong V';\\
0& \text{otherwise}.
\end{cases}
\end{displaymath}
This implies that the associative algebra describing a block of $\mathcal{O}$ is standardly stratified in the 
sense of \cite{CPS2}. In particular, we have the following BGG-reciprocity  in $\mathcal{O}$:
\begin{displaymath}
(P(V):{\Delta}(V'))=[\overline{\nabla}(V'):L(V)]
\end{displaymath}
(compare with Corollary~\ref{Cor17}). Another consequence is that $\mathcal{O}$ has tilting modules
(in the sense of \cite{Fr}). Tilting modules are modules which have both a standard filtration and a
proper costandard filtration. Both $\mathrm{Res}_{\mathfrak{g}_{\overline{0}}}^{\mathfrak{g}}$ and 
$\mathrm{Ind}_{\mathfrak{g}_{\overline{0}}}^{\mathfrak{g}}$ map tilting modules to tilting modules.
It follows that tilting modules are also cotilting.

We refer the reader to \cite{Fr2} for a very detailed $\mathfrak{q}_n$-example where 
${\Delta}(V)$ and $\overline{\Delta}(V)$ are explicitly computed and compared.

\subsection{Dominance dimension and Soergel's Struktursatz}\label{s9.4}

For category $\mathcal{O}$ one has a direct analogue of all our results from Subsection~\ref{s4.3}. 
We claim that every projective $P$ in $\mathcal{O}$ admits a two step coresolution
\begin{displaymath}
0\to P\to X_1\to X_2,
\end{displaymath}
where both $X_1$ and $X_2$ are projective-injective (compare with Proposition~\ref{prop4.3.1}). 
It is certainly enough to prove this for $P(V)$. From the above
we know that the latter module occurs as a direct summand of a module induced from a projective module in 
$\mathfrak{g}_{\overline{0}}$. Since induction is exact and maps projectives to projectives and injectives to 
injectives, the claim follows from the corresponding classical claim for $\mathcal{O}_{\overline{0}}$. 

Pick a representative in each isomorphism class of indecomposable projective-injective modules in 
$\mathcal{O}$ and let $\cC^{PI}$ be the full subcategory of $\mathcal{O}$ formed by these. 
Taking homomorphisms in 
$\mathcal{O}$ into these representatives defines a contravariant functor $\Phi$ from the additive category
of projective objects in $\mathcal{O}$ into $\mathrm{mod}\text{-}\cC^{PI}$
(see Subsection~\ref{s4.3} for more details).

\begin{thm}\label{thm93}
The functor $\Phi$ is fully faithful.
\end{thm}

\begin{proof}
Mutatis mutandis the proof of Theorem~\ref{thm4.3.2}. 
\end{proof}

\subsection{Irving's theorems}\label{s9.6}

We also have the following natural super-analogue of Theorem~\ref{Thm3.9.1}.

\begin{thm}\label{thm95}
Let $V$ be a simple $\mathfrak{h}_0$-dia\-gonal\-izable $\mathfrak{h}$-supermodule. Then the 
following assertions are equivalent:
\begin{enumerate}[$($a$)$]
\item\label{thm95-1} $P(V)$ is isomorphic to $I(V)$ up to parity shift.
\item\label{thm95-2} $L(V)$ occurs in the socle of a projective-injective module in $\mathcal{O}$.
\item\label{thm95-3} $L(V)$ occurs in the top of a projective-injective module in $\mathcal{O}$.
\item\label{thm95-4} $L(V)$ occurs in the socle of some $\Delta(V')$.
\end{enumerate}
\end{thm}

\begin{proof}
This follows by adjunction from the corresponding properties of $\mathcal{O}_{\overline{0}}$.
\end{proof}

\subsection{Constructive differences}\label{s9.7}

The most important difference between category $\mathcal{O}_q$ in the quantum case
and category $\mathcal{O}$ in the super case is that
for the latter we do not know of any analogue for our decomposition statements Theorem~\ref{prop3},
Theorem~\ref{Thm13}, Theorem~\ref{Thm14} and Corollary~\ref{cor4.4.2}. It seems that it is unreasonable to
expect such analogues. To some extent the situation in the super case is opposite to the quantum case.
In the quantum case we have one nice subcategory (the special block) and all other subcategories are more
complicated, however, they are ``more complicated in the same way''. For Lie superalgebras the situation is the
opposite: generically, a Lie superalgebra ``behaves'' like its even part (see e.g. \cite{Go2,FM}), that is
in an ``easy'' way. However, there are non-generic degenerations of various kinds (e.g. atypical modules,
non-diagonalizable $\mathfrak{h}$-supermodules, projective-injectives which are self-dual only up to
a parity shift etc.) and for these degenerations the behavior is more complicated, however, in different ways. 
Therefore we do not think that our tensor decomposition statements have any chance to generalize to the super case.

\vspace{0.3cm}

\noindent
H.~H.~A.: Center for Quantum Geometry of Moduli Spaces, Aarhus 
University, Building 530, Ny Munkegade, 8000  Aarhus C, DENMARK,
email: {\tt mathha\symbol{64}qgm.au.dk}
\vspace{0.3cm}

\noindent
V.~M: Department of Mathematics, Uppsala University, Box. 480,
SE-75106, Uppsala, SWEDEN, email: {\tt mazor\symbol{64}math.uu.se}

\end{document}